\documentclass{amsart}
\usepackage{amsfonts}
\usepackage{amsmath,amssymb}
\usepackage{amsthm}
\usepackage{amscd}
\usepackage{graphics}
\usepackage{graphicx}

\theoremstyle{remark}{
\newtheorem{Def}{{\rm Definition}}
\newtheorem{Ex}{{\rm Example}}
\newtheorem{Rem}{{\rm Remark}}
\newtheorem{Prob}{{\rm Problem}}

}
\theoremstyle{plain}
{

\newtheorem{Prop}{Proposition}
\newtheorem{Thm}{Theorem}

}

\begin{document}
\title[Arrangements of small circles for Morse-Bott functions]{Arrangements of small circles for Morse-Bott functions and regions surrounded by them}
\author{Naoki kitazawa}

\address{Institute of Mathematics for Industry, Kyushu University, 744 Motooka, Nishi-ku Fukuoka 819-0395, Japan\\
 TEL (Office): +81-92-802-4402 \\
 FAX (Office): +81-92-802-4405 \\
}
\email{n-kitazawa@imi.kyushu-u.ac.jp, naokikitazawa.formath@gmail.com}
\urladdr{https://naokikitazawa.github.io/NaokiKitazawa.html}
\maketitle
\begin{abstract}
As a topic of mathematics, "arrangements", systems of hyperplanes, circles, and general (regular) submanifolds, attract us strongly.  We present a natural elementary study of arrangements of circles. It is also a kind of new studies. Our study is closely related to geometry and singularity theory of Morse(-Bott) functions. 
Regions surrounded by circles are regarded as images of real algebraic maps and composing them with projections gives {\it Morse-Bott} functions: this observation is natural, and surprisingly, recently presented first, by the author.

 We present a systematic way of constructing such arrangements by choosing small circles centered at existing circles inductively. We are interested in graphs the regions surrounded by the circles naturally collapse. We have studied local changes of the graphs in adding these circles. These graphs are essentially so-called {\it Reeb graphs} of the previous Morse-Bott functions: they are spaces of all components of preimages of single points for the functions.

\end{abstract}
\section{Introduction.}
\label{sec:1}

Arrangements have attracted us for a long time. They are families of regular submanifolds in the given spaces. For example, arrangements of hyperplanes and subspaces of vector spaces are actively studied. Generalizations to more general regular submanifolds are natural. Studies on these generalizations are still new.

It is very difficult to understand and explain related studies all. \cite{tamaki2, tamaki3} are from webpages \cite{tamaki1}, explaining various mathematics, especially, algebraic topology. They are managed by Dai Tamaki and supported by various mathematicians. \cite{tamaki2, tamaki3} explains arrangements. We are concerned with circles. It is also very difficult to understand and explain the whole world of related studies. For example, \cite{carmesinschulz} is a related study and 
{\cite[Section 1 Related work]{carmesinschulz}} also implies that it is difficult to explain the whole world of related studies precisely.

Our interest lies in arrangements closely related to singularity theory of differentiable maps and applications to various geometry of manifolds, seeming to be a kind of new studies surprisingly. We are interested in the family of circles and the region surrounded by the circles. Furthermore, the region is seen as the image of a natural real algebraic map such that its composition with a projection to a line is a so-called {\it Morse-Bott} function. As further interest, we are trying to have explicit families of such explicit functions, including natural projections of naturally embedded sphere or so-called round spheres as simplest cases. 

The present study is on fundamental studies of the arrangements. We define a certain class: {\it arrangements of circles for Morse-Bott functions}, {\it MB circle arrangements} or {\it MBC arrangements}. We also study explicit construction. 

We present related differential topology and real algebraic geometry with singularity theory shortly again. Of course we do not need to understand related arguments and studies. Existence of nice smooth, analytic, or real algebraic functions and explicit construction of them are different. The former is solved by methods in differential equation and homotopy theory, for example. The latter is rather difficult where it is of course natural, fundamental, and important in geometry. It is difficult to find explicit nice maps of a certain class generalizing certain nice classes of so-called {\it Morse} functions and natural projections of round spheres. For example, for manifolds represented as products of two spheres, slight generalizations obtained as so-called bundles, or manifolds obtained by connected them naturally (or considering so-called connected sums), such construction has been difficult. The author have founded related construction (\cite{kitazawa1, kitazawa2, kitazawa3}). We omit precise exposition on related further topics in the present paper and leave them as another problem. For related real algebraic geometry, see \cite{kollar}, surveying real algebraic geometry closely related to some differential topology. See also a textbook \cite{bochnakcosteroy} for real algebraic geometry.
For related studies on elementary real algebraic geometry and construction of real algebraic functions and maps by the author, see \cite{kitazawa5, kitazawa6, kitazawa7} for example. We do not need to understand them of course.
\subsection{Regions surrounded by our arrangements of circles are naturally images of natural real algebraic maps and compositions with projections are Morse-Bott functions: observations presented first by the author. }\label{subsec:1.1}
We define our arrangements of circles, essentially coming from our original study \cite{kitazawa5, kitazawa7}. 

Hereafter, let ${\mathbb{R}}^n$ denote the $n$-dimensional Euclidean space, a smooth manifold equipped with the standard Euclidean metric. Let $||x|| \geq 0$ denote the distance between $x \in {\mathbb{R}}^n$ and the origin $0 \in {\mathbb{R}}^n$. Let $\mathbb{R}:={\mathbb{R}}^1$. Let ${\pi}_{m,n,i}:{\mathbb{R}}^m \rightarrow {\mathbb{R}}^n$ with $m>n \geq 1$ be the canonical projection mapping $x=(x_1,x_2) \in {\mathbb{R}}^{n} \times {\mathbb{R}}^{m-n}={\mathbb{R}}^m$ to $x_i$ for $i=1,2$.
In addition, for a subspace $Y \subset X$ of a topological space $X$, let $\overline{Y}$ denote the closure considered in the space $X$. 
Let $Y^{\circ}$ denote the interior of $Y$ in $X$. In most cases, we can guess the underlying space $X$ easily. We omit related notation or exposition unless otherwise stated. For a manifold $X$ whose boundary is non-empty, let $\partial X$ denote the boundary.  
\begin{Def}
\label{def:1}
We choose a circle $S _{x_0,r}:=\{x \mid ||x-x_0||=r\} \subset {\mathbb{R}}^2$ with $x_0=(x_{0,1},x_{0,2}) \in {\mathbb{R}}^2$ and $r>0$. We call the points $(x_{0,1} \pm r,x_{0,2})$ {\it horizontal} poles.  We call the points $(x_{0,1},x_{0,2} \pm r)$ {\it vertical} poles.
\end{Def}
\begin{Def}
\label{def:2}
A pair of a family $\mathcal{S}=\{S _{x_j,r_j}:=\{x_j \mid ||x_j-x_{j,0}||=r_j\} \subset {\mathbb{R}}^2\}$ of circles satisfying the following condition and a bounded connected region $D_{\mathcal{S}}$ surrounded by $\mathcal{S}$ is called an {\it arrangement of circles for Morse-Bott functions}. We also call it an {\it MB circle arrangement} or {\it MBC arrangement}. 
\begin{enumerate}
\item $D_{\mathcal{S}} \subset {\mathbb{R}}^2$ is a bounded connected component of the complement ${\mathbb{R}}^2-{\bigcup}_{S _{x_j,r_j} \in \mathcal{S}} S _{x_j,r_j}$ and $\overline{D_{\mathcal{S}}} \bigcap S _{x_j,r_j}$ is not empty for each circle $S _{x_j,r_j} \in \mathcal{S}$.
\item Two distinct circles $S_{x_{i_1},r_{i_1}}$ and $S_{x_{i_2},r_{i_2}}$ from $\mathcal{S}$ intersect in $\overline{D_{\mathcal{S}}}$ obeying the following: for each point $p_{i_1,i_2}$ in such an intersection, the sum of the tangent vector spaces of $S_{x_{i_1},r_{i_1}}$  and $S_{x_{i_2},r_{i_2}}$ at $p_{i_1,i_2}$ coincides with the tangent vector space of ${\mathbb{R}}^2$ at $p_{i_1,i_2}$ and distinct three circles from $\mathcal{S}$ do not intersect there. In addition, these points in these intersections are not vertical poles or horizontal poles of the circles from $\mathcal{S}$.

\end{enumerate}
\end{Def}

With additional data of some positive integers assigned to circles, we have a natural real algebraic map onto $\overline{D_{\mathcal{S}}}$. We present this shortly according to \cite{kitazawa7}.

The set $D_{\mathcal{S}}$ is represented as the intersection ${\bigcap}_{S_{x_j,r_j} \in \mathcal{S}} \{x \mid  f_j(x)>0\}$ where $S_{x_j,r_j}$ is the zero set of a suitable polynomial $f_j$ of degree $2$. 
For each $S_{x_j,r_j} \in \mathcal{S}$, we assign an element of a finite set. As a result we have a map. We redefine this as a surjection by considering a suitable subset of it. As another important condition, at two distinct circles $S_{x_{j_1},r_{j_1}}$ and $S_{x_{j_2},r_{j_2}}$ intersecting at some points in the closure $\overline{D_{\mathcal{S}}}$ of the regions, the values of the map are always distinct.
Let the surjection be denoted by $m_{\mathcal{S},A}: \mathcal{S} \rightarrow A$.
For each $a \in A$, we assign a positive integer $i(a)$. We can define the zero set $M_{D_{\mathcal{S}},m_{\mathcal{S},A},i,f}:=\{(x,y) \in {\mathbb{R}}^2 \times {\prod}_{a \in A} {\mathbb{R}}^{i(a)+1} \mid  {\prod}_{S_{x_j,r_j} \in {m_{\mathcal{S},A}}^{-1}(a)} (f_{j}(x))- {\Sigma}_{j_a=1}^{i(a)} {y_{a,j_a}}^2=0\}$ ($y_a$ is the component of $y \in {\prod}_{a \in A} {\mathbb{R}}^{i(a)+1}$ corresponding to $a \in A$ and $y_{a,j_a}$ denotes the $j_a$-th component). The projection to $\overline{D_{\mathcal{S}}} \subset {\mathbb{R}}^2$ gives our desired map $f_{D_{\mathcal{S}},m_{\mathcal{S},A},i,f}:M_{D_{\mathcal{S}},m_{\mathcal{S},A},i,f} \rightarrow {\mathbb{R}}^2$. The composition with ${\pi}_{2,1,i}:{\mathbb{R}}^2 \rightarrow \mathbb{R}$ is a so-called {\it Morse-Bott} function. We can naturally generalize this construction for general pairs of the dimensions and general hypersurfaces (we consider regions surrounded by hypersurfaces in a general Euclidean space).

Hereafter, $D^k:=\{x \in {\mathbb{R}}^k \mid ||x|| \leq 1\}$ is the $k$-dimensional unit disk and $S^k:=\{x \in {\mathbb{R}}^{k+1} \mid ||x||=1\}$ is the $k$-dimensional unit sphere.
\begin{Ex}\label{ex:1}
We consider the case $\mathcal{S}=\{S^1\}$ and $D_{\mathcal{S}}={(D^2)}^{\circ}$. We have the so-called canonical projection of the unit sphere $S^2$ into ${\mathbb{R}}^2$. See FIGURE \ref{fig:1}, presented later.
\end{Ex}
\subsection{A Poincar\'e-Reeb graph, our region surrounded by the circles collapses.}
\label{subsec:1.2}
We can define a graph for $(\mathcal{S},D_{\mathcal{S}})$ in the following way. We consider the space of all connected components of all preimages of all single points for the function ${\pi}_{2,1,i} {\mid}_{D_{\mathcal{S}}}$, the restriction of the projection. We topologize the set with the quotient topology of the closure $\overline{D_{\mathcal{S}}}$. We can naturally regard this as a graph where we present rigorous rules for vertices later. We call the graph a {\it Poincar\'e-Reeb graph} of $D_{\mathcal{S}}$, respecting \cite{bodinpopescupampusorea, kohnpieneranestadrydellshapirosinnsoreatelen, sorea1, sorea2}. For the case of Example \ref{ex:1}, check FIGURE \ref{fig:1}.

This is also regarded as a so-called {\it Reeb graph} of a natural Morse-Bott function, obtained as the function ${\pi}_{2,1,i} \circ f_{D_{\mathcal{S}},m_{\mathcal{S},A},i,f}:M_{D_{\mathcal{S}},m_{\mathcal{S},A},i,f} \rightarrow \mathbb{R}$. The {\it Reeb graph} $W_c$ of a smooth function $c$ is the space of all connected components of all preimages of all single points for the smooth function $c$ and topologized similarly. We also present rigorously later. This is a classical object in theory of Morse functions, more general nice smooth functions and applications to theory of manifolds. Reeb graphs compactly represent the manifolds. \cite{reeb} is a related pioneering study. 
\subsection{A short presentation of our main result: a list of local changes of Poincar\'e-Reeb graphs by changing MBC arrangements with additions of small circles.}
\label{subsec:1.3}
In short, we have a kind of complete lists of local changes of Poincar\'e-Reeb graphs by changing a given MBC arrangement into another MBC arrangement by an addition of a (generic) small circle.

The next section is mainly for preliminaries. We discuss our main result (Theorems \ref{thm:1}--\ref{thm:5}) in the third section.
\section{Preliminaries.}
\label{sec:2}
\subsection{Smooth manifolds and maps.}
\label{subsec:2.1}
Consider a topological space naturally represented as a cell complex such as a (topological) manifold, a polyhedron, a graph, and a polyhedron. The dimension is a topological invariant for such spaces and non-negative integers. Let $\dim X$ denote the dimension of $X$.
For a smooth manifold $X$, let $T_xX$ denote the tangent vector space at $x \in X$. Let $c:X \rightarrow Y$ be a differentiable map from a differentiable manifold $X$ into another differentiable manifold $Y$. Let ${dc}_x:T_x X \rightarrow T_{c(x)}Y$ denote the differential of $c$ at $x \in X$: this is a linear map between the tangent vector spaces. If the rank of ${dc}_x$ is smaller than both the dimensions $\dim X$ and $\dim Y$, then $x$ is a {\it singular point} of $c$.
A {\it diffeomorphism} on a smooth manifold is a smooth map being also a homeomorphism having no singular points. The space of all smooth maps between two smooth manifolds is endowed with the so-called {\it Whitney $C^{\infty}$ topology}. We do not explain this topology precisely. In short, this respects not only the values of the maps but also the derivations of the maps globally. \cite{golubitskyguillemin} presents related introductory theory and (classical) advanced theory. Here, we do not need to understand Morse(-Bott) functions or general related singularity theory. However, we only note that a smooth function $c:X \rightarrow \mathbb{R}$ on a manifold $X$ with no boundary is a Morse function if at each singular point $p$ of the function it is locally represented by the map mapping $(x_1,\cdots, x_m)$ to ${\Sigma}_{j=1}^{i(p)} {x_j}^2-{\Sigma}_{j=1}^{m-i(p)} {x_{i(p)+j}}^2+c(p)$ for suitable local coordinates and a suitable (uniquely defined) integer $0 \leq i(p) \leq m$. We also note that a {\it Morse-Bott} function is a smooth function such that at each singular point of it the function is represented as the composition of a {\it submersion}, a smooth map which has no singular point, with a Morse function. Of course a Morse function is Morse-Bott.

We do not need to understand real algebraic geometry. For related survey, see \cite{kollar}. We only note that our real algebraic manifold $M$ is a union of connected components of the zero set of a real polynomial map and in some Euclidean space ${\mathbb{R}}^m$. Our real algebraic map is the canonical embedding of the set into the Euclidean space or the composition of this with the projection ${\pi}_{m,n,i}$. Furthermore, our real algebraic manifolds are also {\it non-singular}: this is defined via the implicit function theorem for the real polynomial map.   

\subsection{Several graphs: Reeb graphs and Poincar\'e-Reeb graphs again.}\label{subsec:2.2}
A {\it graph} is a $1$-dimensional CW complex. $0$-cells are {\it vertices} and the set of all $0$-cells is the {\it vertex set} of the graph.  
$1$-cells are {\it edges} and the set of all $1$-cells is the {\it edge set} of the graph. We can orient the edge set or each edge of the graph. We call such a graph an {\it oriented graph} or a {\it digraph}.
We can label the vertex set or each vertex by real numbers for a graph.  We call such a graph a {\it V-graph}.
An {\it oriented V-graph} means an oriented graph compatible with the values assigned to the vertices: in other words, if an edge connecting two vertices $v_1$ and $v_2$ is oriented as an edge departing from $v_1$ and entering $v_2$, then the value at $v_2$ is greater than the value at $v_1$. For example, this means that a loop, or an edge from a vertex to the same vertex, is not contained in any oriented V-graph. 

For two oriented V-graphs, we can naturally define an {\it isomorphism}. A piecewise smooth map between two oriented V-graphs is an {\it isomorphism} if the following are satisfied.
\begin{itemize}
\item The map is a piecewise smooth homeomorphism mapping a vertex set into another vertex set. By this only, the map is an {\it isomorphism} of the graphs and the graphs are {\it isomorphic}. 
\item The map preserves orientations of edges. By this and the previous condition only, the map is an {\it isomorphism} of the digraphs and the digraphs are {\it isomorphic}. 
\item The map preserves the orders of values at vertices.
\end{itemize}
We also say that they are {\it isomorphic} as oriented V-graphs or V-digraphs.

For a graph, the {\it degree} of a vertex means the number of edges containing the vertex.

The {\it Reeb space} $W_c$ of a smooth function $c:X \rightarrow \mathbb{R}$ is the quotient space $X/{\sim c}$ defined by the equivalence relation ${\sim}_c$: we define as $x_1 \sim x_2$ if and only if they belong to a same connected component of a same preimage $c^{-1}(y)$. Let $q_c:X \rightarrow W_c$ denote the quotient map. By defining the set of all vertices as follows, this is regarded as a graph in considerable cases. For example, the case the set of all values realized as values at some singular points of $c$ is finite is one of such cases. This is shown in \cite{saeki1}. This explains cases of Morse-Bott functions for example. We consider these cases unless otherwise stated. 
A point $v$ in $W_c$ is a vertex if and only if ${q_f}^{-1}(v)$ contains some singular points of the function $c$. 
We call this graph the {\it Reeb graph} of $c$.
We can define a natural continuous function $\bar{c}:W_c \rightarrow \mathbb{R}$ by the relation $c=\bar{c} \circ q_c$ in a unique way. We can naturally orient the edges. We can label each vertex by the value of the function there.
We have an oriented V-graph and call it the {\it Reeb V-digraph} of $c$.

Hereafter, we abuse the notation in Definition \ref{def:2} and "Subsection \ref{subsec:1.1}". The {\it Poincar\'e-Reeb graph of an MBC arrangement $(\mathcal{S},D_{\mathcal{S}})$ or the region $D_{\mathcal{S}}$ defined for ${\pi}_{2,1,i}$} is the quotient space of $\overline{D_{\mathcal{S}}}/{\sim}_{{\pi}_{2,1,i}}$ defined by the equivalence relation ${\sim}_{{\pi}_{2,1,i}}$: we define $x_1 {\sim}_{{\pi}_{2,1,i}} x_2$ if and only if they belong to a same connected component of a same preimage ${{\pi}_{2,1,i}}^{-1}(y)$. Let $q_{D_{\mathcal{S}},i}:\overline{D_{\mathcal{S}}} \rightarrow \overline{D_{\mathcal{S}}}/{\sim}_{{\pi}_{2,1,i}}$ denote the quotient map. A {\it pre-vertex} $v$ is defined as a point whose preimage ${q_{D_{\mathcal{S}},i}}^{-1}(v)$ contains vertical poles of circles in $\mathcal{S}$ or intersections of distinct circles in $\mathcal{S}$ for $i=1$.  A {\it pre-vertex} $v$ is defined as a point whose preimage ${q_{D_{\mathcal{S}},i}}^{-1}(v)$ contains horizontal poles of circles in $\mathcal{S}$ or intersections of distinct circles in $\mathcal{S}$ for $i=2$.
It seems to be clear that the resulting space has the structure of a graph. However, we give rigorous exposition for this. 

Hereafter, we assume elementary knowledge on (fiber) bundles.

\begin{Prop}
	\label{prop:1}
	The quotient space $\overline{D_{\mathcal{S}}}/{\sim}_{{\pi}_{2,1,i}}$ is a graph by defining the vertex set as the set of all pre-vertices here. We call this graph the {\rm Poincar\'e-Reeb graph of an MBC arrangement $(\mathcal{S},D_{\mathcal{S}})$ or the region $D_{\mathcal{S}}$ defined for ${\pi}_{2,1,i}$}.
\end{Prop}
\begin{proof}
We give very general proof applying some theory from \cite{saeki1, saeki2}.
First, the nature of circles, our assumption on their intersections, and compactness imply that the set of all pre-vertices is finite. We respect this implicitly and argue.

For any point $t \in \overline{D_{\mathcal{S}}}/{\sim}_{{\pi}_{2,1,i}}$ which is not a vertex and its small neighborhood $N(t)$, consider the preimage ${q_{D_{\mathcal{S}},i}}^{-1}(N(t))$ and remove the preimage ${q_{D_{\mathcal{S}},i}}^{-1}(t)$ of the chosen point. The resulting space is divided into two connected components, represented as trivial bundles over intervals whose projections are defined by the projection $q_{D_{\mathcal{S}},i}$ and whose fibers are diffeomorphic to the closed interval $D^1$. Here we apply arguments of (a variant of) Ehresmann's theorem. For a vertex $t$, we consider the same argument and we can argue similarly and have a similar result. We have finitely many connected components, represented as trivial bundles over intervals whose projections are defined by the projection $q_{D_{\mathcal{S}},i}$ and whose fibers are diffeomorphic to the closed interval $D^1$. 

For example, \cite[Theorem 2.12]{saeki2} implies that for the cornered smooth compact manifold $\overline{D_{\mathcal{S}}}$ and the map ${\pi}_{2,1,i} {\mid}_{\overline{D_{\mathcal{S}}}}$, the quotient space has the structure of a graph defined in the presented way.
\end{proof}

Remark \ref{rem:1} is our summary on facts and arguments we implicitly use hereafter.
A {\it smooth} bundle means a bundle whose fiber is a smooth manifold and whose structure group is a group of diffeomorphisms on the fiber.
In addition, hereafter, a (straight) segment or line is {\it horizontal} ({\it vertical}) if it is represented as a connected subset of the straight line by the form $\{(x_1,t) \mid x_1 \in \mathbb{R}\}$ (resp. $\{(t,x_2) \mid x_2 \in \mathbb{R}\}$) for a real number $t$. The case $t=0$ is for the axis (axes).
\begin{Rem} 
	\label{rem:1}
	We use Ehresmann's theorem in our proof of Proposition \ref{prop:1}. It implies that a smooth surjection with no singular point is the projection of a smooth bundle if it is {\it proper}: a {\it proper} map is a map between topological spaces such that the preimage of any compact set is compact.

	In our case, each fiber $F$ is a vertical segment of the form $\{(p_{F,0},t_2) \mid p_{F,1} \leq t_2 \leq p_{F,2}\}$ connecting one point $(p_{F,0}, p_{F,1})$ in the boundary $\overline{D_{\mathcal{S}}}$ and another point $(p_{F,0}, p_{F,2})$ there. This also gives a smooth bundle over any interval regarded as an interval embedded in the interior of an edge.
	
	Moreover, the interior of the vertical segment consists of points of the interior $\overline{D_{\mathcal{S}}}-D_{\mathcal{S}}$.

	Furthermore, we consider the complementary set of the set of all points contained in exactly two circles from $\mathcal{S}$ in $\overline{D_{\mathcal{S}}}-D_{\mathcal{S}}$. This manifold ${\partial}_{\rm s} {D_{\mathcal{S}}}$ is regarded as a $1$-dimensional smooth manifold. The points $(p_{F,0}, p_{F,1})$ and $(p_{F,0}, p_{F,2})$ are in the interior of the $1$-dimensional smooth manifold. Here the sum of the tangent space of $F$ at $p$ and the tangent space of ${\partial}_{\rm s} {D_{\mathcal{S}}}$ at $p$ coincides with the tangent space of ${\mathbb{R}^2}$ at $p$.
	
	As another remark, of course, we can use more geometric way to understand that our space of Proposition \ref{prop:1} is a graph. In short, for any (converging) sequence in the base space of a smooth bundle whose fiber is diffeomorphic to $F$ and $D^1$, converges to a point of the base space, regarded as an interval in the space $\overline{D_{\mathcal{S}}}/{\sim}_{{\pi}_{2,1,i}}$, or a pre-vertex of the space $\overline{D_{\mathcal{S}}}/{\sim}_{{\pi}_{2,1,i}}$.
\end{Rem}
Let $W_{D_{\mathcal{S}},i}$ denote the Poincar\'e-Reeb graph $\overline{D_{\mathcal{S}}}/{\sim}_{{\pi}_{2,1,i}}$ defined for ${\pi}_{2,1,i}$. We can also naturally orient the edges in this case. We can also label each vertex by the value of the function similarly: we can define a natural unique continuous function $m_{D_{\mathcal{S}},i}:W_{D_{\mathcal{S}},i} \rightarrow \mathbb{R}$ by the relation same as one for the case of Reeb graphs. We have an oriented V-graph and call it a {\it Poinar\'e-Reeb V-digraph} of $D_{\mathcal{S}}$. 

In the previous section, we have introduced that the composition of a natural real algebraic map $f_{D_{\mathcal{S}},m_{\mathcal{S},A},i,f}:M_{D_{\mathcal{S}},m_{\mathcal{S},A},i,f} \rightarrow {\mathbb{R}}^2$ onto the closure $\overline{D_{\mathcal{S}}}$ of the region $D_{\mathcal{S}}$ with ${\pi}_{2,1,i}:{\mathbb{R}}^2 \rightarrow \mathbb{R}$ is a Morse-Bott function. We can check that this V-digraph $W_{D_{\mathcal{S}},i}$ and the Reeb V-digraph $W_{{\pi}_{2,1,i} \circ f_{D_{\mathcal{S}},m_{\mathcal{S},A},i,f}}$ of the corresponding Morse-Bott function are isomorphic. Hereafter, we do not need or use such arguments on such Morse-Bott functions. Our paper concerns arrangements of circles, essentially.

FIGURE \ref{fig:1} is for the case of Example \ref{ex:1}.
 \begin{figure}
\includegraphics[width=80mm,height=27.5mm]{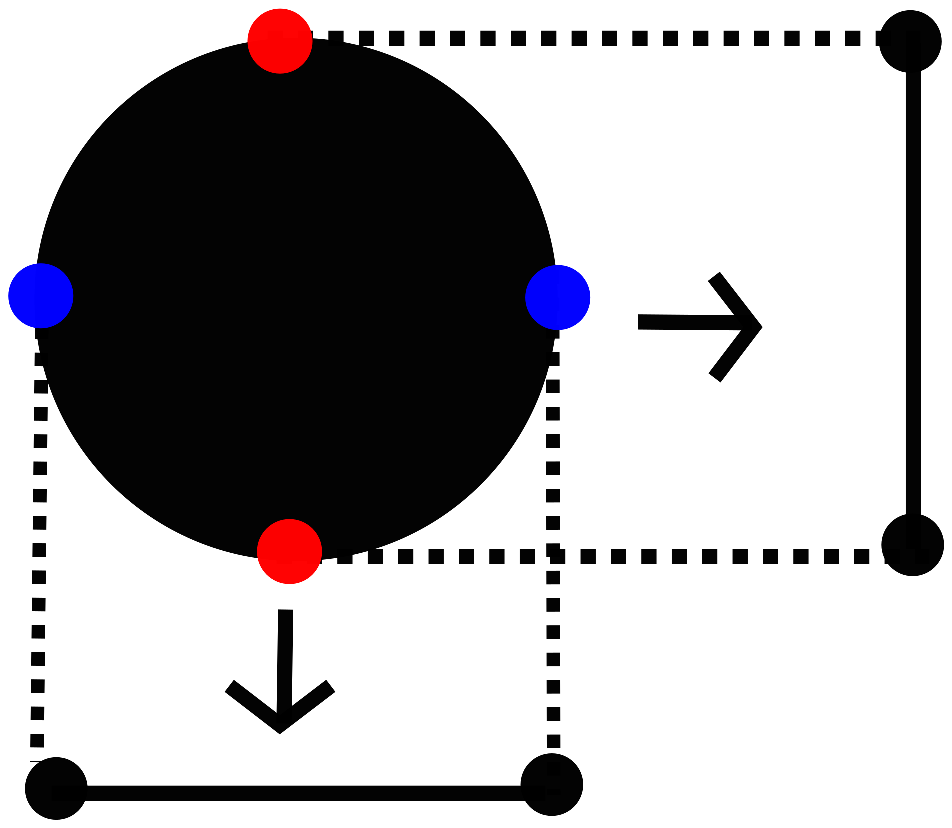}
\caption{The case of Example \ref{ex:1}. Blue dots show vertical poles and red dots show horizontal poles. Graphs are Poinar\'e-Reeb graphs of $D_{\mathcal{S}}$.}
\label{fig:1}
\end{figure}
\subsection{Some observations on circle arrangements.}
\label{subsec:2.3}
Some elementary transformations produce MBC arrangements which are different from the original ones and which are "congruent" to them. We present this property as follows. We can easily check from fundamental properties of the transformations.
\begin{Prop}
\label{prop:2}
\begin{enumerate}
	\item \label{prop:2.1} By choosing two real numbers $x_{0,1},x_{0,2} \in \mathbb{R}$ and mapping each $(x_1,x_2)$ to $(x_1+x_{0,1},x_2+x_{0,2})$, we have another MBC arrangement. In other words, we consider a parallel transformation. This preserves the Poincar\'e-Reeb V-digraphs $W_{D_{\mathcal{S}},i}$ up to isomorphisms.
\item \label{prop:2.2} By choosing $x_0 \in \mathbb{R}$ and mapping each $(x_1,x_2)$ to $(2x_0-x_1,x_2)$ or mapping each $(x_1,x_2)$ to $(x_1,2x_0-x_2)$, we have another MBC arrangement. In other words, we consider reflections around vertical lines and horizontal lines, respectively. This preserves the Poincar\'e-Reeb V-digraphs $W_{D_{\mathcal{S}},i}$ up to isomorphisms.
\item \label{prop:2.3} By choosing $(x_{0,1},x_{0,2}) \in {\mathbb{R}}^2$ and mapping each $(x_1,x_2)$ to $(x_{0,1}-(x_2-x_{0,2}),x_{0,2}+(x_1-x_{0,1}))$ or mapping each $(x_1,x_2)$ to $(x_{0,1}+(x_2-x_{0,2}),x_{0,2}-(x_1-x_{0,1}))$, we have another MBC arrangement. In other words, we consider rotations of degree $\pm \frac{\pi}{2}$ centered at the fixed single points.
\end{enumerate}
\end{Prop}
Remark \ref{rem:2} is also a kind of remarks on arguments we implicitly use.
\begin{Rem}
	\label{rem:2}
	We assume fundamental facts and arguments on elementary plane geometry and Euclidean geometry. For example, we need fundamental geometry on circles in the plane.	
	
	As another case, we consider subspaces of metric spaces with the induced metrics. Explicitly, we regard circles and more general subspaces as subspaces (with the induced metrics) of the Euclidean space ${\mathbb{R}}^2$ (with the standard Euclidean metric).
	\end{Rem}

\section{On our main result.}
\label{sec:3}
\subsection{A class of circle arrangements: circle-centered circle arrangements.}
\label{subsec:3.1}
\begin{Def}
\label{def:3}
A pair of a set $\mathcal{S}:=\{S_{x_{j},r_{j}}\}$ of circles of fixed radii in ${\mathbb{R}}^2$ and a region $D_{\mathcal{S}} \subset {\mathbb{R}}^2$ which is also one of a bounded connected component of the complementary set ${\mathbb{R}}^2-{\bigcup}_{S_{x_{j},r_{j}} \in \mathcal{S}} S_{x_{j},r_{j}}$ is called a {\it circle-centered} arrangement if we can have $\mathcal{S}$ inductively in the following way.
\begin{itemize}
\item Choose a set ${\mathcal{S}}_0$ of mutually disjoint circles $S_{x_{j,0},r_{j,0}} \in {\mathcal{S}}_0$ of fixed radii in ${\mathbb{R}}^2$ such that we can choose a bounded region $D_{{\mathcal{S}}_0}$ surrounded by all of them. In other words, the boundary $\partial \overline{D_{{\mathcal{S}}_0}}$ of the closure $\overline{D_{{\mathcal{S}}_0}}$ is the disjoint union of all circles $S_{x_{j,0},r_{j,0}}$. Of course the set of circles is non-empty.
\item Let ${\mathcal{S}}:={\mathcal{S}}_0$.
\item We consider the following procedure inductively.

First we choose a point $x_{j^{\prime}} \in S_{x_j,r_j}$ in a circle of $\mathcal{S}$. Second we choose a sufficiently small circle of a fixed radius $r_{j^{\prime}}>0$ which is centered at $x_{j^{\prime}}$. We define the new set ${\mathcal{S}}$ by adding the new element $S_{x_{j^{\prime}},r_{j^{\prime}}}$ to the original set. We also define the new region $D_{\mathcal{S}}$ as the intersection of the present region and the complement of the closed disk whose boundary is $S_{x_{j^{\prime}},r_{j^{\prime}}}$ in ${\mathbb{R}}^2$.
\end{itemize} 
\end{Def}
\subsection{The statement and its proof.}
\label{subsec:3.2}
\begin{Thm}
\label{thm:1}
\begin{enumerate}
\item \label{thm:1.1} Circle-centered arrangements are MBC arrangements{\rm :} we can call such arrangements {\rm MB circle-centered arrangements},  or {\rm MBCC arrangements}.
\item \label{thm:1.2} For an MBC arrangement $(\mathcal{S}:=\{S_{x_{j},r_{j}}\},D_{\mathcal{S}})$, we first choose a point $p=x_{j^{\prime}} \in S_{x_j,r_j} \bigcap \overline{D_{\mathcal{S}}}$
 in a circle of $\mathcal{S}$. We choose a sufficiently small circle of a fixed radius $r_{j^{\prime}}>0$ which is centered at $x_{j^{\prime}}$. We define the set ${\mathcal{S^{\prime}}}$ as the new set adding the new element $S_{x_{j^{\prime}},r_{j^{\prime}}}$ to the original set. We also define the new region $D_{\mathcal{S^{\prime}}}$ as the intersection of the present region and the complement of the closed disk whose boundary is $S_{x_{j^{\prime}},r_{j^{\prime}}}$ in ${\mathbb{R}}^2$. In other words, this is same as the procedure presented in Definition \ref{def:3}.
Then the resulting new pair $({\mathcal{S^{\prime}}},D_{\mathcal{S^{\prime}}})$ of the set of circles and the region is also an MBC arrangement. If the given MBC arrangement is an MBCC arrangement, then the resulting one is also an MBCC arrangement. 
\end{enumerate}
\end{Thm}

We present most important part of our result. Remember notation on Poincar\'e-Reeb graphs (V-digraphs) in our "Subsection \ref{subsec:2.2}". Hereafter, as a kind of elementary operations, we add vertices in the interiors of edges for example. In such procedures, we preserve the originally given orientations of the edges and preserve real numbers used for labeling vertices for V-digraphs, for example, unless otherwise stated. For example, in Theorem \ref{thm:2} (\ref{thm:2.2}), we need to change  real numbers for labeling vertices slightly. 
\begin{Thm}
	\label{thm:2}
	Consider an operation of Theorem \ref{thm:1}, creating a new pair $({\mathcal{S^{\prime}}},D_{\mathcal{S^{\prime}}})$ of a family of circles and a region surrounded by the circles from an MBC arrangement $(\mathcal{S}:=\{S_{x_{j},r_{j}}\},D_{\mathcal{S}})$. By this operation, a Poincar\'e-Reeb V-digraph $W_{D_{\mathcal{S^{\prime}}},i}$ of the new region $D_{\mathcal{S^{\prime}}}$ is changed from the Poincar\'e-Reeb V-digraph $W_{D_{\mathcal{S}},i}$ of the original one $D_{\mathcal{S}}$ as follows{\rm :} we concentrate on the case $i=1$ here and we can have the similar result for $i=2$ by the symmetry or Proposition \ref{prop:2} {\rm (}Proposition \ref{prop:2} {\rm (}\ref{prop:2.3}{\rm )}{\rm )}.

\begin{enumerate}
\item \label{thm:2.1}  Let $p=x_{j^{\prime}}$ be a point which is contained in the 
unique circle $S_{x_{j},r_{j}}$ from $\mathcal{S}$ and which is not a vertical pole or a horizontal pole.
In this case, we can argue in either of the following ways.
\begin{enumerate}
\item \label{thm:2.1.1} We choose a suitable oriented edge $e_0$ of the V-digraph $W_{D_{\mathcal{S}},1}$ and we add two distinct vertices $v_1$ and $v_2$ in its interior and another new oriented edge $e$ connecting $v_1$ and another new vertex $v$. Furthermore, we choose these vertices and edges so that the edge connecting $v_1$ and $v_2$ and contained as a subset in $e_0$ and $e$ are both departing from $v_1$ or both entering $v_1$. 
In addition, we choose them in such a way that the resulting values $m_{D_{\mathcal{S}},1}(v_1)$ and $m_{D_{\mathcal{S}},1}(v_2)$ are sufficiently close and satisfying the relation  $m_{D_{\mathcal{S}},1}(v_1)<{\pi}_{2,1,1}(p)<m_{D_{\mathcal{S}},1}(v_2)$ or $m_{D_{\mathcal{S}},1}(v_2)<{\pi}_{2,1,1}(p)<m_{D_{\mathcal{S}},1}(v_1)$.
We give the orientation of $e$ by labeling $v$ by a real number $i(v)$ sufficiently close to $m_{D_{\mathcal{S}},1}(v_1)$ with either the relation $m_{D_{\mathcal{S}},1}(v_1)<i(v)<{\pi}_{2,1,1}(p)<m_{D_{\mathcal{S}},1}(v_2)$ or $m_{D_{\mathcal{S}},1}(v_2)<{\pi}_{2,1,1}(p)<i(v)<m_{D_{\mathcal{S}},1}(v_1)$.
The Poincar\'e-Reeb V-digraph $W_{D_{\mathcal{S^{\prime}}},1}$ is isomorphic to the resulting new $V$-digraph. 
\item \label{thm:2.1.2} We choose suitable two adjacent oriented edges $e_{0,1}$ and $e_{0,2}$ of the V-digraph $W_{D_{\mathcal{S}},1}$ such that $e_{0,1}$ departs from a vertex $v_{0,1}$ and enters a vertex $v_0$ and that $e_{0,2}$ departs from $v_0$ and enters a vertex $v_{0,2}$. We add a vertex $v_i$ in the interior of each edge $e_{0,i}$ and another new oriented edge $e$ connecting one of the vertices $v_1$ and $v_2$ and another new vertex $v$. 
Furthermore, we choose these vertices and edges so that $e$ departs from $v_1$ or that $e$ enters $v_2$. In addition, we choose them in such a way that the resulting values $m_{D_{\mathcal{S}},1}(v_1)$ and $m_{D_{\mathcal{S}},1}(v_2)$ are sufficiently close and satisfying the relation $m_{D_{\mathcal{S}},1}(v_1)<{\pi}_{2,1,1}(p)=m_{D_{\mathcal{S}},1}(v_0)<m_{D_{\mathcal{S}},1}(v_2)$. We also label the vertex $v$ by a real number $i(v)$ sufficiently close to $m_{D_{\mathcal{S}},1}(v_1)$ with the relation $m_{D_{\mathcal{S}},1}(v_1)<i(v)<{\pi}_{2,1,1}(p)=m_{D_{\mathcal{S}},1}(v_0)<m_{D_{\mathcal{S}},1}(v_2)$ if $e$ departs from $v_1$. We also label the vertex $v$ by a real number $i(v)$ sufficiently close to $m_{D_{\mathcal{S}},1}(v_2)$ with the relation $m_{D_{\mathcal{S}},1}(v_1)<{\pi}_{2,1,1}(p)=m_{D_{\mathcal{S}},1}(v_0)<i(v)<m_{D_{\mathcal{S}},1}(v_2)$ if $e$ enters $v_2$. 
 The Poincare-Reeb V-digraph $W_{D_{\mathcal{S^{\prime}}},1}$ is isomorphic to the resulting new $V$-digraph.
\end{enumerate}
\item \label{thm:2.2}  Let $p=x_{j^{\prime}}$ be a point which is contained in the unique circle $S_{x_{j},r_{j}}$ from $\mathcal{S}$ and which is a vertical pole. In this case, we also assume that the unique connected component of the intersection $\{(p_1,t) \mid t \in \mathbb{R}\} \bigcap \overline{D_{\mathcal{S}}}$ of the vertical line $\{(p_1,t) \mid t \in \mathbb{R}\}$ and the closure $\overline{D_{\mathcal{S}}}$ containing $p$ does not contain any vertical poles, or points contained in exactly two circles from $\mathcal{S}$ where $p=(p_1,p_2)$. Under the situation, we can argue in either of the following ways. 
\begin{enumerate}
	\item \label{thm:2.2.1} We choose a suitable oriented edge $e_0$ entering {\rm (}departing from{\rm )} a vertex $v_0$ of degree $1$ with $m_{D_{\mathcal{S}},1}(v_0)={\pi}_{2,1,1}(p)$ in the V-digraph $W_{D_{\mathcal{S}},1}$. We add two new vertices $v_1$ and $v_2$ and two new oriented edges $e_j$ departing from {\rm (}resp. entering{\rm )} $v_0$ and entering {\rm (}resp. departing from{\rm )} $v_j$ {\rm (}$j=1,2${\rm )}. 
Moreover, we label the vertices $v_1$ and $v_2$ by a suitable same real number $i(v_1)=i(v_2)$ smaller {\rm (}resp. greater{\rm )} than $m_{D_{\mathcal{S}},1}(v_0)={\pi}_{2,1,1}(p)$ and sufficiently close to $m_{D_{\mathcal{S}},1}(v_0)$. We also label $v_0$ by a real number smaller {\rm (}resp. greater{\rm )} than $i(v_1)=i(v_2)$ and sufficiently close to $i(v_1)=i(v_2)$ instead of the original value $m_{D_{\mathcal{S}},1}(v_0)$.  
	The Poincare-Reeb V-digraph $W_{D_{\mathcal{S^{\prime}}},1}$ is isomorphic to the resulting new $V$-digraph.
	\item \label{thm:2.2.2} We choose a suitable oriented edge $e_0$ entering {\rm (}departing from{\rm )} a vertex $v_0$ of degree $3$ with $m_{D_{\mathcal{S}},1}(v_0)={\pi}_{2,1,1}(p)$ and two oriented edges $e_{0,1}$ and $e_{0,2}$ departing from {\rm (}resp. entering{\rm )} $v_0$ in the V-digraph $W_{D_{\mathcal{S}},1}$.
	 We add a vertex $v_i$ in the interior of each edge $e_{0,i}$. We label the vertices $v_1$ and $v_2$ by a suitable same real number $i(v_1)=i(v_2)$ greater {\rm (}resp. smaller{\rm )} than $m_{D_{\mathcal{S}},1}(v_0)$ and sufficiently close to $m_{D_{\mathcal{S}},1}(v_0)$. We also label $v_0$ by a real number smaller {\rm (}resp. greater{\rm )} than the original value $m_{D_{\mathcal{S}},1}(v_0)$ and sufficiently close to the value instead of $m_{D_{\mathcal{S}},1}(v_0)={\pi}_{2,1,1}(p)$.
	The Poincare-Reeb V-digraph $W_{D_{\mathcal{S^{\prime}}},1}$ is isomorphic to the resulting new $V$-digraph. 
\end{enumerate}
\item \label{thm:2.3}  Let $p=x_{j^{\prime}}$ be a point which is contained in the unique circle $S_{x_{j},r_{j}}$ from $\mathcal{S}$ and which is a horizontal pole. 

In this case, we can argue in either of the following ways.
\begin{enumerate}
\item \label{thm:2.3.1} We choose a suitable edge $e_0$ departing from a vertex $v_{0,1}$ and entering a vertex $v_{0,2}$ in the V-digraph $W_{D_{\mathcal{S}},1}$. For this, the relation $m_{D_{\mathcal{S}},1}(v_{0,1})<{\pi}_{2,1,1}(p)<m_{D_{\mathcal{S}},1}(v_{0,2})$ must hold. We add two distinct vertices $v_1$ adjacent to $v_{0,1}$ and $v_2$ adjacent to $v_{0,2}$ in the interior of $e_0$. 
We label the vertex $v_1$ by a real number $i(v_1)$ greater than $m_{D_{\mathcal{S}},1}(v_{0,1})$. We label the vertex $v_2$ by a real number $i(v_2)$ smaller than $m_{D_{\mathcal{S}},1}(v_{0,2})$. We also label them in such a way that $i(v_1)$ and $i(v_2)$ are sufficiently close with the relation $i(v_1)<{\pi}_{2,1,1}(p)<i(v_2)$.
The V-digraph $W_{D_{\mathcal{S^{\prime}}},1}$ is isomorphic to the resulting new $V$-digraph.
\item \label{thm:2.3.2} We choose a suitable oriented edge $e_0$ departing from a vertex $v_{0,1}$ and entering a vertex $v_{0,2}$ in the V-digraph $W_{D_{\mathcal{S}},1}$. For this, the relation $m_{D_{\mathcal{S}},1}(v_{0,1})<{\pi}_{2,1,1}(p)<m_{D_{\mathcal{S}},1}(v_{0,2})$ must hold. We add two distinct vertices $v_1$ adjacent to $v_{0,1}$ and $v_2$ adjacent to $v_{0,2}$ in its interior. We label the vertex $v_1$ by a real number $i(v_1)$ greater than $m_{D_{\mathcal{S}},1}(v_{0,1})$. We label the vertex $v_2$ by a real number $i(v_2)$ smaller than $m_{D_{\mathcal{S}},1}(v_{0,2})$. We also label them in such a way that $i(v_1)$ and $i(v_2)$ are sufficiently close with the relation $i(v_1)<{\pi}_{2,1,1}(p)<i(v_2)$. For a vertex $v_1$, we add another oriented edge $e_1$ departing from $v_1$ and entering another new vertex $v_3$. For a vertex $v_2$, we add another oriented edge $e_2$ departing from another new vertex $v_4$ and entering $v_2$. We label $v_3$ and $v_4$ by suitable real numbers $i(v_3)$ and $i(v_4)$ satisfying the relation $i(v_1)<i(v_3)<{\pi}_{2,1,1}(p)<i(v_4)<i(v_2)$. 
The V-digraph $W_{D_{\mathcal{S^{\prime}}},1}$ is isomorphic to the resulting new $V$-digraph.
\item \label{thm:2.3.3} 
We choose suitable two adjacent oriented edges $e_{0,1}$ and $e_{0,2}$ in the graph $W_{D_{\mathcal{S}},1}$ such that $e_{0,1}$ departs from a vertex $v_{0,1}$ and enters a vertex $v_0$ with $m_{D_{\mathcal{S}},1}(v_0)={\pi}_{2,1,1}(p)$ and that $e_{0,2}$ departs from $v_0$ and enters $v_{0,2}$ in the V-digraph $W_{D_{\mathcal{S}},1}$. We add a vertex $v_i$ in the interior of the edge $e_{0,i}$ {\rm (}$i=1,2${\rm )}. We label the vertex $v_1$ by a real number $i(v_1)$ greater than $m_{D_{\mathcal{S}},1}(v_{0,1})$ and sufficiently close to $m_{D_{\mathcal{S}},1}(v_{0})={\pi}_{2,1,1}(p)${\rm :} the relation $m_{D_{\mathcal{S}},1}(v_{0})>i(v_1)$ holds. We label the vertex $v_2$ by a real number $i(v_2)$ smaller than $m_{D_{\mathcal{S}},1}(v_{0,2})$ and sufficiently close to $m_{D_{\mathcal{S}},1}(v_{0})={\pi}_{2,1,1}(p)${\rm :} the relation $m_{D_{\mathcal{S}},1}(v_{0})<i(v_2)$ holds.
The V-digraph $W_{D_{\mathcal{S^{\prime}}},1}$ is isomorphic to the resulting new $V$-digraph.
\item \label{thm:2.3.4} We choose suitable two adjacent oriented edges $e_{0,1}$ and $e_{0,2}$ in the V-digraph $W_{D_{\mathcal{S}},1}$ such that $e_{0,1}$ departs from a vertex $v_{0,1}$ and enters a vertex $v_0$ with $m_{D_{\mathcal{S}},1}(v_0)={\pi}_{2,1,1}(p)$ and that $e_{0,2}$ departs from $v_0$ and enters $v_{0,2}$. We add a vertex $v_i$ in the interior of the edge $e_{0,i}$ {\rm (}$i=1,2${\rm )}. We label the vertex $v_1$ by a real number $i(v_1)$ greater than $m_{D_{\mathcal{S}},1}(v_{0,1})$ and sufficiently close to $m_{D_{\mathcal{S}},1}(v_0)={\pi}_{2,1,1}(p)${\rm :} the relation $m_{D_{\mathcal{S}},1}(v_0)>i(v_1)$ holds. We label the vertex $v_2$ by a real number $i(v_2)$ smaller than $m_{D_{\mathcal{S}},1}(v_{0,2})$ and sufficiently close to $m_{D_{\mathcal{S}},1}(v_0)={\pi}_{2,1,1}(p)${\rm :} the relation $m_{D_{\mathcal{S}},1}(v_0)<i(v_2)$ holds.
For a vertex $v_1$, we add another oriented edge $e_1$ departing from $v_1$ and entering another new vertex $v_3$. For a vertex $v_2$, we add another oriented edge $e_2$ departing from another new vertex $v_4$ and entering $v_2$.  We label $v_3$ and $v_4$ by suitable real numbers $i(v_3)$ and $i(v_4)$ satisfying the relation $m_{D_{\mathcal{S}},1}(v_{0,1})<i(v_1)<i(v_3)<m_{D_{\mathcal{S}},1}(v_0)={\pi}_{2,1,1}(p)<i(v_4)<i(v_2)<m_{D_{\mathcal{S}},1}(v_{0,2})$. 
The V-digraph $W_{D_{\mathcal{S^{\prime}}},1}$ is isomorphic to the resulting new $V$-digraph.

\end{enumerate}
\end{enumerate}
\end{Thm}
\begin{Thm}
	\label{thm:3}
	We consider an operation of Theorem \ref{thm:1}, creating new pair $({\mathcal{S^{\prime}}},D_{\mathcal{S^{\prime}}})$ of a family of circles and a region surrounded by the circles from an MBC arrangement $(\mathcal{S}:=\{S_{x_{j},r_{j}}\},D_{\mathcal{S}})$, again. We have the following.
	\begin{enumerate}	\setcounter{enumi}{3}
		\item \label{thm:3.1}
	
Let $p=x_{j^{\prime}}$ be a point which is contained in exactly two circles $S_{x_{j_1},r_{j_1}}$ and $S_{x_{j_2},r_{j_2}}$ from $\mathcal{S}$. Note that by the definition of an MBC arrangement, this point is not a horizontal pole or a vertical pole of any circle from $\mathcal{S}$. We have the unique two point set represented as the intersection $\overline{D_{\mathcal{S}}} \bigcap S_{x_{j^{\prime}},r_{j^{\prime}}}$ and the union of two one-point sets as $(S_{x_{j^{\prime}},r_{j^{\prime}}} \bigcap S_{x_{j_1},r_{j_1}}) \sqcup (S_{x_{j^{\prime}},r_{j^{\prime}}} \bigcap S_{x_{j_2},r_{j_2}})$.

 We also consider the straight line tangent to $S_{x_{j_1},r_{j_1}}$ at $p$ and the straight line tangent to $S_{x_{j_2},r_{j_2}}$ at $p$ and consider the intersection with $S_{x_{j^{\prime}},r_{j^{\prime}}}$.

For each straight line corresponding to the circle $S_{x_{j_a},r_{j_a}}$ we have a two-point set $\{p_{j_a,j^{\prime},1},p_{j_a,j^{\prime},2}\}$. 
We can choose the unique point $p_{j_a,j^{\prime},b}$ closer to the set $D_{\mathcal{S}} \bigcap S_{x_{j^{\prime}},r_{j^{\prime}}}$ in $S_{x_{j^{\prime}},r_{j^{\prime}}}$ uniquely for each $a=1,2$.

We can form the angle by the two segments each of which connects $p$ and $p_{j_a,j^{\prime},b}$ {\rm (}$a=1,2${\rm )}. We call the angle a {\rm canonical angle at $(p,\overline{D_{\mathcal{S}}})$}{\rm :} we can give the unique orientations to these two segments from $p$ to the points, represent the oriented segments by the form of vectors
	$(v_{p,\overline{D_{\mathcal{S}}},1,1},v_{p,\overline{D_{\mathcal{S}}},1,2})$ and $(v_{p,\overline{D_{\mathcal{S}}},2,1},v_{p,\overline{D_{\mathcal{S}}},2,2})$ with $v_{p,\overline{D_{\mathcal{S}}},a,a^{\prime}} \neq 0$ {\rm (}$a=1,2,a^{\prime}=1,2${\rm )}, call the two oriented segments {\rm canonical angle segments for the angle}, and we can also define the size of the angle by a positive real number $0<\theta<\pi$ canonically. 
	\item \label{thm:3.2}
	By this operation, the Poincar\'e-Reeb V-digraph $W_{D_{\mathcal{S^{\prime}}},1}$ of the new region $D_{\mathcal{S^{\prime}}}$ is changed from the corresponding Poinar\'e-Reeb V-digraph $W_{D_{\mathcal{S}},1}$ of the original one $D_{\mathcal{S}}$ as follows{\rm :} we concentrate on the case $i=1$ here and we can have the similar result for $i=2$ by the symmetry or Proposition \ref{prop:2} {\rm (}Proposition \ref{prop:2} {\rm (}\ref{prop:2.3}{\rm )}{\rm )}.
	\begin{enumerate}
	\item \label{thm:3.2.1}  Suppose that for the canonical angle segments $(v_{p,\overline{D_{\mathcal{S}}},1,1},v_{p,\overline{D_{\mathcal{S}}},1,2})$ and $(v_{p,\overline{D_{\mathcal{S}}},2,1},v_{p,\overline{D_{\mathcal{S}}},2,2})$ for the canonical angle at $(p,\overline{D_{\mathcal{S}}})$, the relations $v_{p,\overline{D_{\mathcal{S}}},1,1}v_{p,\overline{D_{\mathcal{S}}},2,1}>0$ and $v_{p,\overline{D_{\mathcal{S}}},1,2}v_{p,\overline{D_{\mathcal{S}}},2,2}>0$ hold.

	In this case, the resulting Poincar\'e-Reeb V-digraph $W_{D_{\mathcal{S^{\prime}}},1}$ is as follows.
	
	First, in the Poincar\'e-Reeb V-digraph $W_{D_{\mathcal{S}},1}$, we choose a suitable edge $e_0$ entering {\rm (}departing from{\rm )} a vertex $v_0$ of degree $1$ and add a vertex $v$ in its interior.
	
We label the vertex $v_0$ by a smaller {\rm (}resp. greater{\rm )} real number $i(v_0)$ being sufficiently close to $m_{D_{\mathcal{S}},1}(v_0)$ instead.
	
	 We also label the vertex $v$ by a real number $i(v)$ smaller {\rm (}resp. greater{\rm )} than $i(v_0)$ and sufficiently close to it. The resulting Poincar\'e-Reeb V-digraph $W_{D_{\mathcal{S^{\prime}}},1}$ is isomorphic to this.

\item \label{thm:3.2.2} For the numbers above, suppose that the relations $v_{p,\overline{D_{\mathcal{S}}},1,1}v_{p,\overline{D_{\mathcal{S}}},2,1}>0$ and $v_{p,\overline{D_{\mathcal{S}}},1,2}v_{p,\overline{D_{\mathcal{S}}},2,2}<0$ hold. In this case, the resulting Poincar\'e-Reeb V-digraph is as either of the following.
\begin{enumerate}
\item \label{thm:3.2.2.1} As presented in Theorem \ref{thm:2} {\rm (}\ref{thm:2.2.1}{\rm )} with the relation $i(v_1)=i(v_2)$ being dropped.
\item \label{thm:3.2.2.2} As presented in Theorem \ref{thm:2} {\rm (}\ref{thm:2.2.1}{\rm )}.
\end{enumerate}

\item \label{thm:3.2.3} For the numbers above, suppose that the relations $v_{p,\overline{D_{\mathcal{S}}},1,1}v_{p,\overline{D_{\mathcal{S}}},2,1}<0$ and $v_{p,\overline{D_{\mathcal{S}}},1,2}v_{p,\overline{D_{\mathcal{S}}},2,2}>0$ hold. In this case, the resulting Poincar\'e-Reeb V-digraph is as presented in Theorem \ref{thm:2} {\rm (}\ref{thm:2.3.1}{\rm )} or Theorem \ref{thm:2} {\rm (}\ref{thm:2.3.3}{\rm )}.

\item \label{thm:3.2.4} For the numbers above, suppose that the relations $v_{p,\overline{D_{\mathcal{S}}},1,1}v_{p,\overline{D_{\mathcal{S}}},2,1}<0$ and $v_{p,\overline{D_{\mathcal{S}}},1,2}v_{p,\overline{D_{\mathcal{S}}},2,2}<0$ hold. In this case, the resulting Poincar\'e-Reeb V-digraph is as presented in Theorem \ref{thm:2} {\rm (}\ref{thm:2.3.1}{\rm )} or Theorem \ref{thm:2} {\rm (}\ref{thm:2.3.3}{\rm )}.

\end{enumerate}
\end{enumerate}
\end{Thm}

\begin{proof}[A proof of Theorems \ref{thm:1}--\ref{thm:3}]
First, we assume several facts and arguments and use them implicitly. Remember Remarks \ref{rem:1} and \ref{rem:2} for example.

We first show Theorem \ref{thm:2}. After that we show Theorem \ref{thm:3} explicitly. We do not assume statements or important arguments related to Theorem \ref{thm:1} of course. 
After the investigation, we have Theorem \ref{thm:1} immediately. 

In fact, Theorem \ref{thm:2} and Theorem \ref{thm:3} are shown by a kind of routine work.

They are shown by local observations of circles around the chosen point $p$ centered at a point in some circle $S_{x_j,r_j} \in \mathcal{S}$. FIGURE \ref{fig:2} shows a complete list of local observations for $p=x_{j^{\prime}} \in S_{x_{j^{\prime}},r_{j^{\prime}}}$. We assume the list. Note also that some are related by transformations
 of Proposition \ref{prop:2}. We also give exposition respecting the transformations.  

\begin{figure}
	\includegraphics[width=80mm,height=80mm]{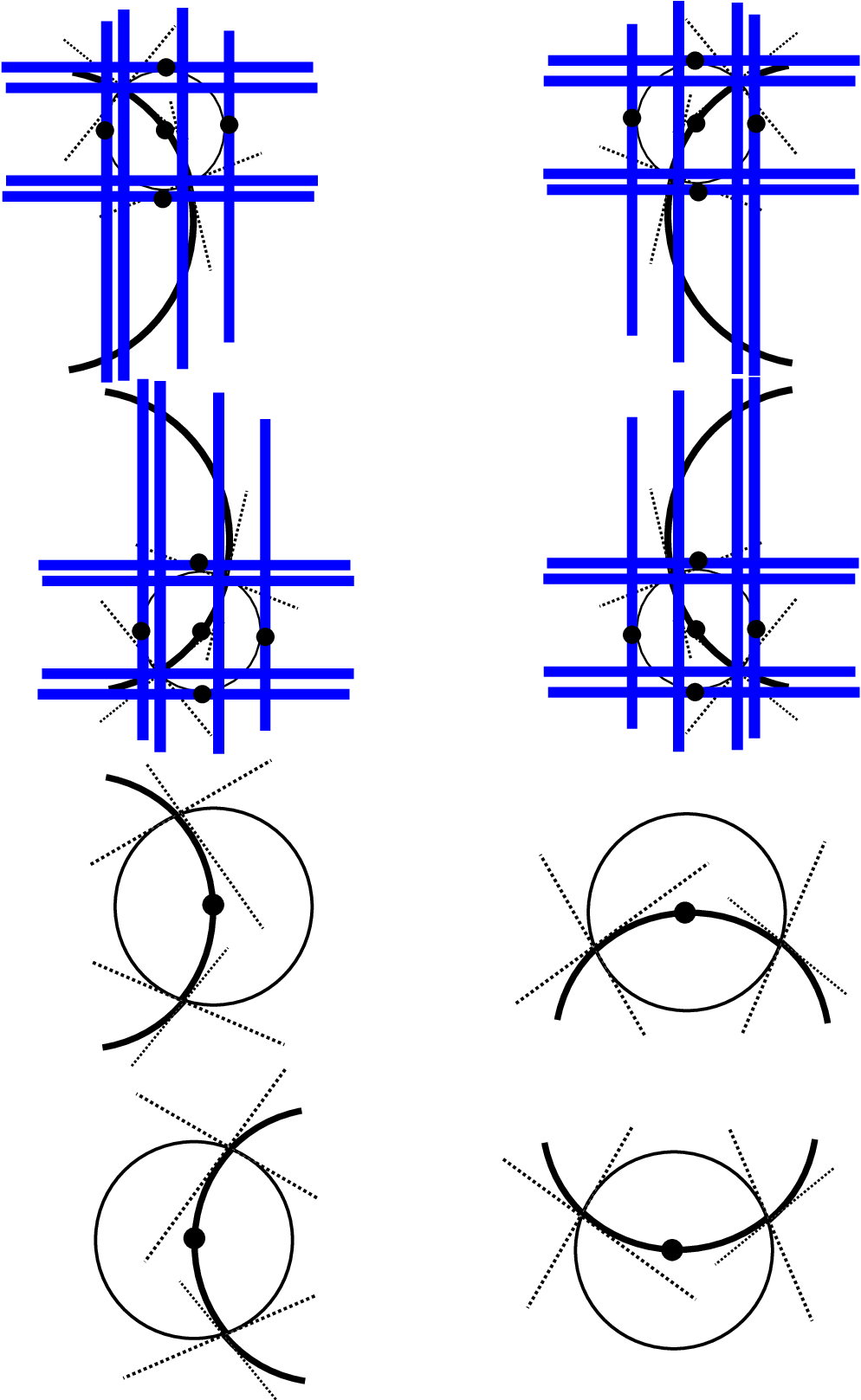}
	\caption{Local observations for $p=x_{j^{\prime}} \in S_{x_{j^{\prime}},r_{j^{\prime}}}$. Blue colored lines show vertical lines and horizontal lines. Dots are for vertical poles, horizontal poles, and the centers of circles. Dotted lines are for straight lines tangent to circles at points contained in exactly two circles $S_{x_{j},r_{j}} \in \mathcal{S}$ and $S_{x_{j^{\prime}},r_{j^{\prime}}} \in \mathcal{S^{\prime}}$.
Furthermore, each of the first four is for a case where $p$ is not a vertical pole or a horizontal pole. In addition, the remaining four is for a case where $p$ is a vertical pole or a horizontal pole: we omit horizontal lines and vertical lines here.}
	\label{fig:2}
\end{figure}

We assume several facts and arguments and use them implicitly. Remember Remarks \ref{rem:1} and \ref{rem:2} for example.

As another important remark, a new circle $S_{x_{j^{\prime}},r_{j^{\prime}}}$ is centered at a point in some existing circle $S_{x_j,r_j} \in \mathcal{S}$ and sufficiently small. The sufficiently small circle yields the property that values at vertices of V-graphs are sufficiently close.

We check Theorem \ref{thm:2}.

The case (\ref{thm:2.1}) is already discussed in \cite{kitazawa7} partially. 

The first case (\ref{thm:2.1.1}) is for the case $p=x_{j^{\prime}}$ is not mapped to any vertex by the quotient map $q_{D_{\mathcal{S}},i}$. See related figures \cite[FIGUREs 1 and 2]{kitazawa7} for example. However, we argue this case again here.

We consider the case where after a suitable operation of Proposition \ref{prop:2} (\ref{prop:2.1}), we have the case $p$ is in the circle centered at the origin $0$ and $p=(p_1,p_2)$ with $p_1,p_2>0$. The intersection of the new small disk bounded by the circle $S_{x_{j^{\prime}},r_{j^{\prime}}}$ and $\overline{D_{\mathcal{S}}}$ is mapped to a closed arc contained in the interior of an edge. By considering the location of the circles, we can change our circles, region and Poincar\'e-Reeb graph locally as presented in the statement. For local structures of our Poincar\'e-Reeb V-digraphs of the present case, we can also refer to also FIGURE \ref{fig:2}, presented later. 
Note that different from the present case, the point $p=x_{j^{\prime}}$ is mapped to a vertex by the quotient map $q_{D_{\mathcal{S}},i}$ in FIGURE \ref{fig:2} and that FIGURE \ref{fig:2} is for the case (\ref{thm:2.1.2}).

We can consider other cases for the signs of $p_1$ and $p_2$ in $p$, by the symmetry. More precisely, in our paper, we consider operations of Proposition \ref{prop:2} (\ref{prop:2.2}). As another way of the argument, we can also respect FIGURE \ref{fig:2} for the other cases. This completes our proof of the case (\ref{thm:2.1.1}).

Hereafter, we discuss cases not being discussed in \cite{kitazawa7}, only. FIGUREs are used for explicit cases mainly. 

We explain the second case (\ref{thm:2.1.2}). We can consider the case $p=x_{j^{\prime}}$ is mapped to a vertex $v_p$ by the quotient map $q_{D_{\mathcal{S}},i}$. For example, FIGURE \ref{fig:3} explicitly shows the case by an explicit global example.
\begin{figure}
\includegraphics[width=80mm,height=50mm]{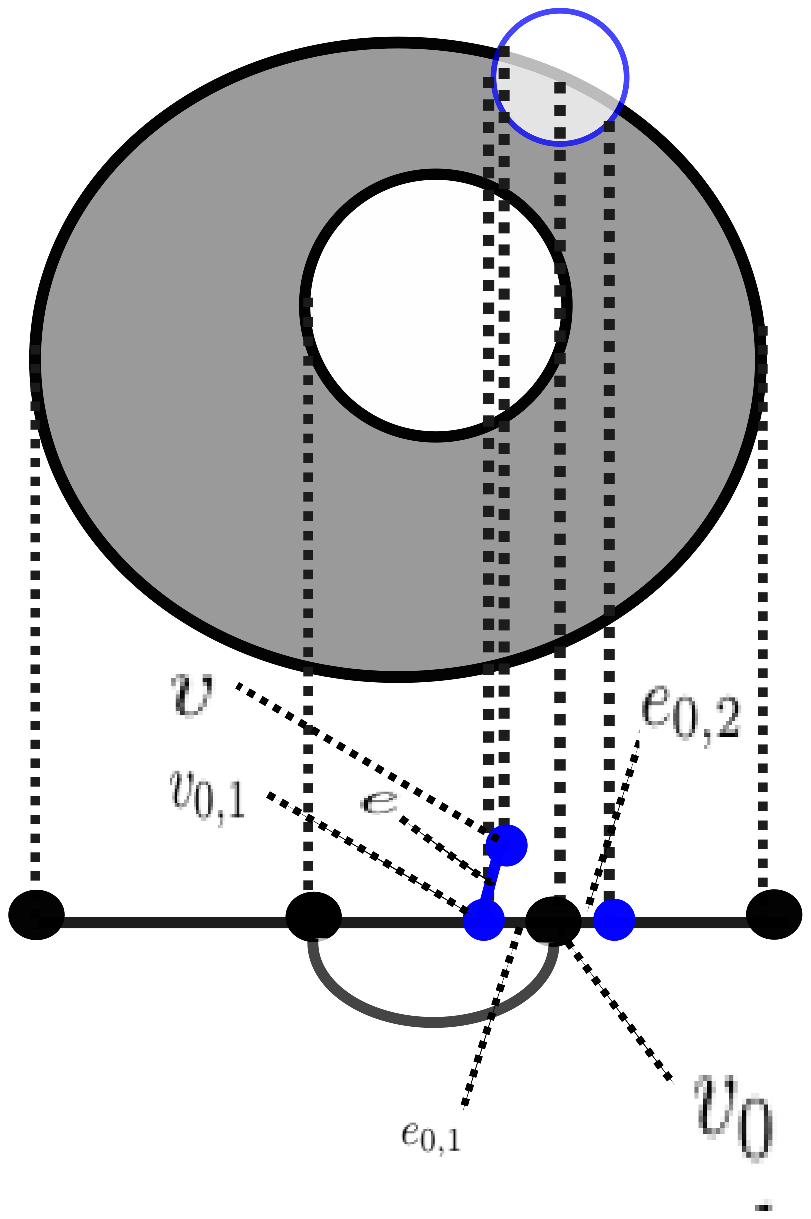}
\caption{A case for the case (\ref{thm:2.1.2}). The blue circle shows the circle $S_{x_{j^{\prime}},r_{j^{\prime}}}$ and the change of our graph is depicted in blue. Hereafter, we use colored objects in this way. Vertices and edges are shown by the notation with dotted segments.}
\label{fig:3}
\end{figure}
 We argue in a general way.
The intersection of the new small disk bounded by the circle $S_{x_{j^{\prime}},r_{j^{\prime}}}$ and $\overline{D_{\mathcal{S}}}$ is mapped to a closed arc containing exactly one vertex so that the vertex is $v_p$ and in the interior of the closed arc. 
Furthermore, the image, represented as the closed arc, is oriented consistently and naturally from the existing oriented edges. 
By considering the location of the circles, we can change similarly except the vertex $v_p$.

We discuss the case (\ref{thm:2.2}).

We discuss the case (\ref{thm:2.2.1}). This case explains the explicit case of FIGURE \ref{fig:4} where the example depicts globally. By considering the location of the circles, we can change our circles, region and Poincar\'e-Reeb graph locally as presented in the statement.
\begin{figure}
\includegraphics[width=80mm,height=50mm]{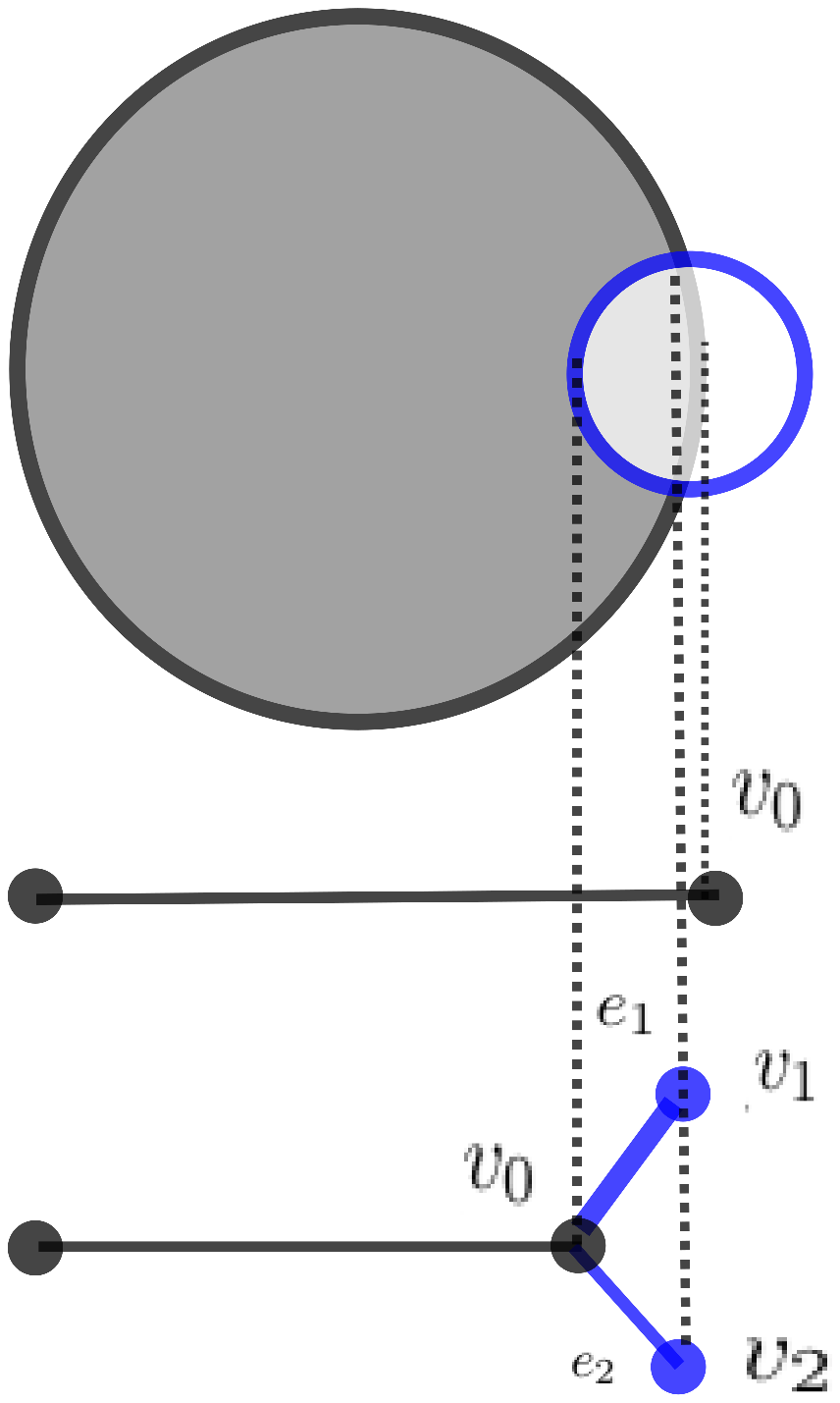}
\caption{The case (\ref{thm:2.2.1}): vertices and edges are shown by the notation.}
\label{fig:4}
\end{figure}

We discuss the case (\ref{thm:2.2.2}). This case explains the explicit case of FIGURE \ref{fig:5} where the example depicts globally. By considering the location of the circles, we can change our circles, region and Poincar\'e-Reeb graph locally as presented in the statement.
\begin{figure}
\includegraphics[width=80mm,height=50mm]{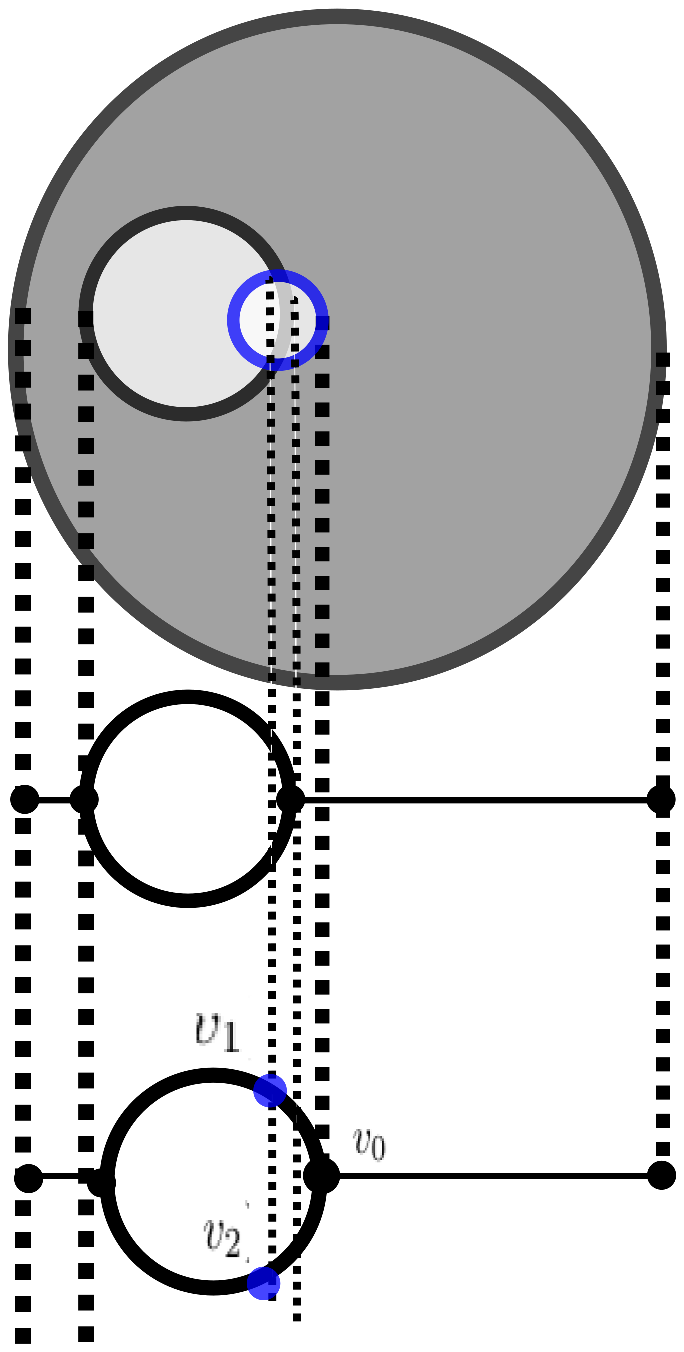}
\caption{The case (\ref{thm:2.2.2}): vertices and edges are shown by the notation.}
\label{fig:5}
\end{figure}


We discuss the case (\ref{thm:2.3}).

The cases (\ref{thm:2.3.1}) and (\ref{thm:2.3.2}) are for the case $p=x_{j^{\prime}}$ is not mapped to any vertex by the quotient map $q_{D_{\mathcal{S}}.1}$. 
For related explicit cases, see FIGURE \ref{fig:6} where these cases show simplest global examples.
We consider the case where after a suitable operation of Proposition \ref{prop:2} (\ref{prop:2.1}), we have the case $p$ is in the circle centered at the origin $0$ and $p=(0,p_2)$ with $p_1=0$ and $p_2>0$ or $p_2<0$. The set $D_S$ may be located above $p$ or beyond $p$. In FIGURE \ref{fig:6}, the first case gives a case of (\ref{thm:2.3.1}) with $p_2<0$ and the second case gives a case of (\ref{thm:2.3.2}) with $p_2>0$.
By considering the location of the circles, we can change our circles, region and Poincar\'e-Reeb graph locally as presented in the statement. 
 This completes our proof of the cases (\ref{thm:2.3.1}) and \ref{thm:2.3.2}.

The cases (\ref{thm:2.3.3}) and (\ref{thm:2.3.4}) are for the case $p=x_{j^{\prime}}$ is mapped to a vertex $v_p$ by the quotient map $q_{D_{\mathcal{S}},1}$. As we do in the case of (\ref{thm:2.1}), or more precisely, (\ref{thm:2.1.2}), we can change similarly except the vertex $v_p$.

\begin{figure}
\includegraphics[width=80mm,height=80mm]{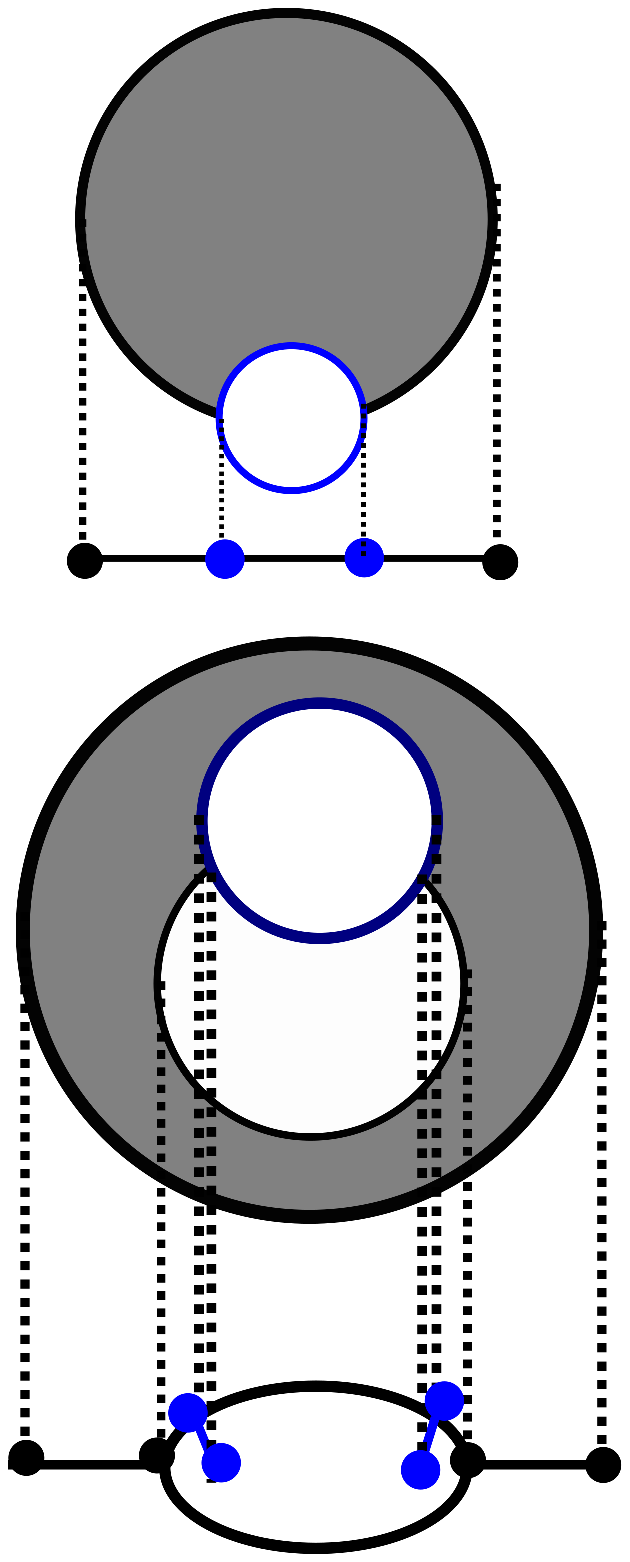}
\caption{Explicit cases for the case (\ref{thm:2.3}). We omit the notation on vertices and edges.}
\label{fig:6}
\end{figure}

This completes our proof of Theorem \ref{thm:2}.

We show Theorem \ref{thm:3} (\ref{thm:3.1}).
The straight lines tangent to the circles $S_{x_{j_1},r_{j_1}}$ and $S_{x_{j_2},r_{j_2}}$ at $p$ are not parallel to vertical or horizontal lines. They intersect in $p$ and form an angle the size of which is $0<\theta<\pi$.
This is thanks to the following and our assumptions and definitions.
\begin{itemize}
	\item Each point contained in exactly two circles $S_{j_1}$ and $S_{j_2}$ from $\mathcal{S}$ is not a vertical point or a horizontal point. 
	\item These circles $S_{j_1}$ and $S_{j_2}$ intersect in such a way that the sum of their tangent vector spaces there coincides with the tangent vector space of ${\mathbb{R}}^2$ there.

\end{itemize}
	The set $D_{\mathcal{S}}$ is represented as an intersection of the form ${\bigcap}_{S_{x_j,r_j} \in \mathcal{S}} \{x \mid f_j(x)>0\}$ where $S_{x_j,r_j}$ is the zero set of a suitably chosen polynomial $f_j$ of degree $2$. We can choose the unique point $p_{j_a,j^{\prime},b}$ closer to the intersection $D_{\mathcal{S}}\bigcap S_{x_{j^{\prime}},r_{j^{\prime}}}$ or contained there as a more specific case.

	For the understanding, we can also consider a straight line $L_{p,D_{\mathcal{S}}}$ being different from the two lines tangent to the circles at $p$, passing through $p$, and locating the set represented as the intersection of the closed disk bounded by $S_{x_{j^{\prime}},r_{j^{\prime}}}$ and $\overline{D_{\mathcal{S}}}-\{p\}$ as a subset of one of the connected components of ${\mathbb{R}}^2-L_{p,D_{\mathcal{S}}}$.
	By the assumption that the two straight lines tangent to the circles at $p$ are not parallel to horizontal lines or vertical lines, we can have desired vectors $(v_{p,\overline{D_{\mathcal{S}}},1,1},v_{p,\overline{D_{\mathcal{S}}},1,2})$ and $(v_{p,\overline{D_{\mathcal{S}}},2,1},v_{p,\overline{D_{\mathcal{S}}},2,2})$ with $v_{p,\overline{D_{\mathcal{S}}},a,a^{\prime}} \neq 0$ where $a=1,2$ and $a^{\prime}=1,2$. They can be also represented as the differences $p_{j_a,j^{\prime},b}-p$ or generally $t(p_{j_a,j^{\prime},b}-p)$ ($t>0$) where the points are seen as $2$-dimensional vectors canonically. 
	
	See also FIGURE \ref{fig:7}.
	
	 \begin{figure}
		\includegraphics[width=80mm,height=80mm]{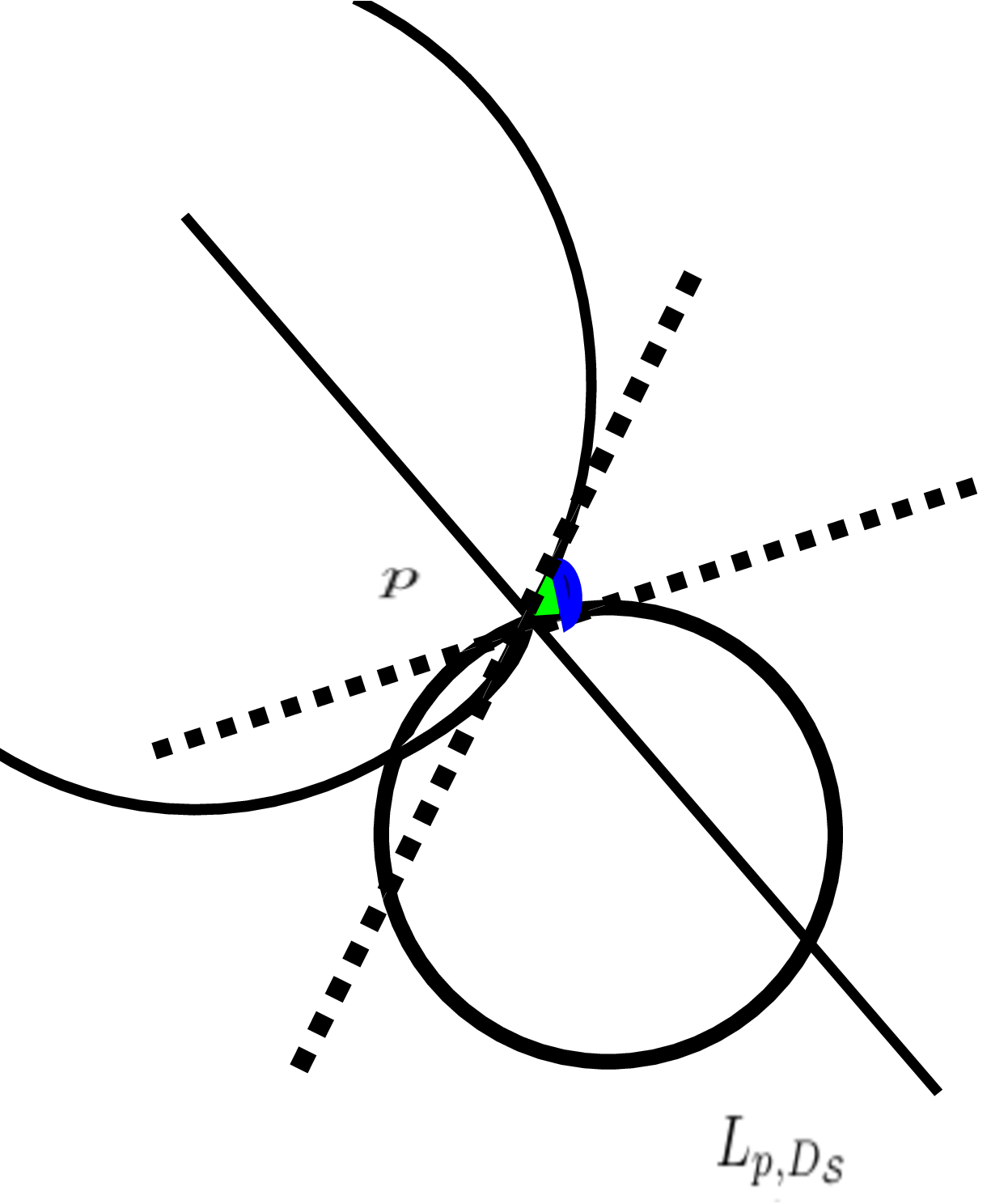}
		\caption{A local picture around a point $p \in S_{x_{j_1},r_{j_1}} \bigcap S_{x_{j_2},r_{j_2}}$ (for the case (\ref{thm:3.2.1})). The set of all points of $D_{\mathcal{S}}$ which are also contained in the interior of the closed disk bounding the new circle $S_{X_{j^{\prime}},r_{j^{\prime}}}$, depicted partially as an arc colored in blue, is colored in green. Dotted lines are for the straight lines tangent to the two circles at $p$.}
		\label{fig:7}
	\end{figure}
	
	We show Theorem \ref{thm:3} (\ref{thm:3.2}). We can check (\ref{thm:3.2.1}) by the locations of circles. For an explicit case, see FIGURE \ref{fig:7} again. The remaining cases are considered based on whether the signs of the $i$-th components of the two canonical angle segments for the canonical angle at $(p,\overline{D_{\mathcal{S}}})$ are same or not ($i=1,2$). For these remaining cases, we can check simply by the locations of circles. More precisely, by respecting FIGURE \ref{fig:2} and curvatures of the circles for example, we can reach our new result.

	This completes the proof of Theorem \ref{thm:3}.
	
	Except more general cases for Theorem \ref{thm:2} (\ref{thm:2.2}), we have considered and investigated all possible cases for the operation of creating $({\mathcal{S}}^{\prime},D_{\mathcal{S}^{\prime}})$ from $(\mathcal{S},D_{\mathcal{S}})$.
	
	In our investigation, we can also check Theorem \ref{thm:1} easily. Note that we do not assume statements or related important arguments of Theorem \ref{thm:1} in our proof of Theorems \ref{thm:2} and \ref{thm:3}.
	
	This completes the proof.
\end{proof}
The following is a kind of additional small comments.
\begin{Thm}
\label{thm:4}
Consider a procedure of constructing MBCC arrangements one after another as in Definition \ref{def:3} or Theorems \ref{thm:1}--\ref{thm:3}.

In Theorem \ref{thm:3}, either our local change presented in {\rm (}\ref{thm:3.2.1}{\rm )}, one presented in {\rm (}\ref{thm:3.2.2.2}{\rm )}, or one presented in {\rm (}\ref{thm:3.2.4}{\rm )} does not occur.

\end{Thm}
\begin{proof}
By our proof of Theorems \ref{thm:1}--\ref{thm:3} and discussions on the local changes, 
we cannot encounter the case $(\ref{thm:3.2.1})$ or $(\ref{thm:3.2.4})$. Check FIGURE \ref{fig:2}, presenting local circles and lines around $p=x_{j^{\prime}}$ for example.

We show the non-existence of the case (\ref{thm:3.2.2.2}). If such a case occurs, then by an elementary geometry of circles in the Euclidean plane, some distinct circles $S_{x_1,r_1}$ and $S_{x_2,r_2}$ in our family of circles are of a same radius. They intersect in a two-point set. However, the definition of an MBCC arrangement forces us to choose one of the circle here must be centered at a point in the other circle and of a sufficiently small radius. This is a contradiction.

This completes our proof.
\end{proof}
\subsection{Related examples and remarks.}
\label{subsec:3.3}
\begin{Ex}
\label{ex:2}
This is related to \cite[Main Theorems 4 and 5 and FIGURE 8]{kitazawa7}. We consider the case Example \ref{ex:1} with $(\mathcal{S}=\{S_{x_1,r_1}=S^1\},D_{\mathcal{S}})$. We add a circle as presented in "FIGURE 8 there".
We have an MBC arrangement as presented in "our new FIGURE \ref{fig:8}". This is an MBC arrangement. In short, we choose a new circle $S_{x_2,r_2}$ around a small segment connecting two points in the given circle $S_{x_1,r_1}$ in such a way that the new circle passes through the two points and is centered at a point outside the given circle and that the arc in the new circle is sufficiently close to the small segment above.
We have ${\mathcal{S}}^{\prime}=\{S_{x_1,r_1}, S_{x_2,r_2}\}$. The new region $D_{{\mathcal{S}}^{\prime}}$ is the intersection of $D_{\mathcal{S}}$ and the complementary set of the closed disk bounded by $S_{x_2,r_2}$.
\begin{figure}
	\includegraphics[width=80mm,height=50mm]{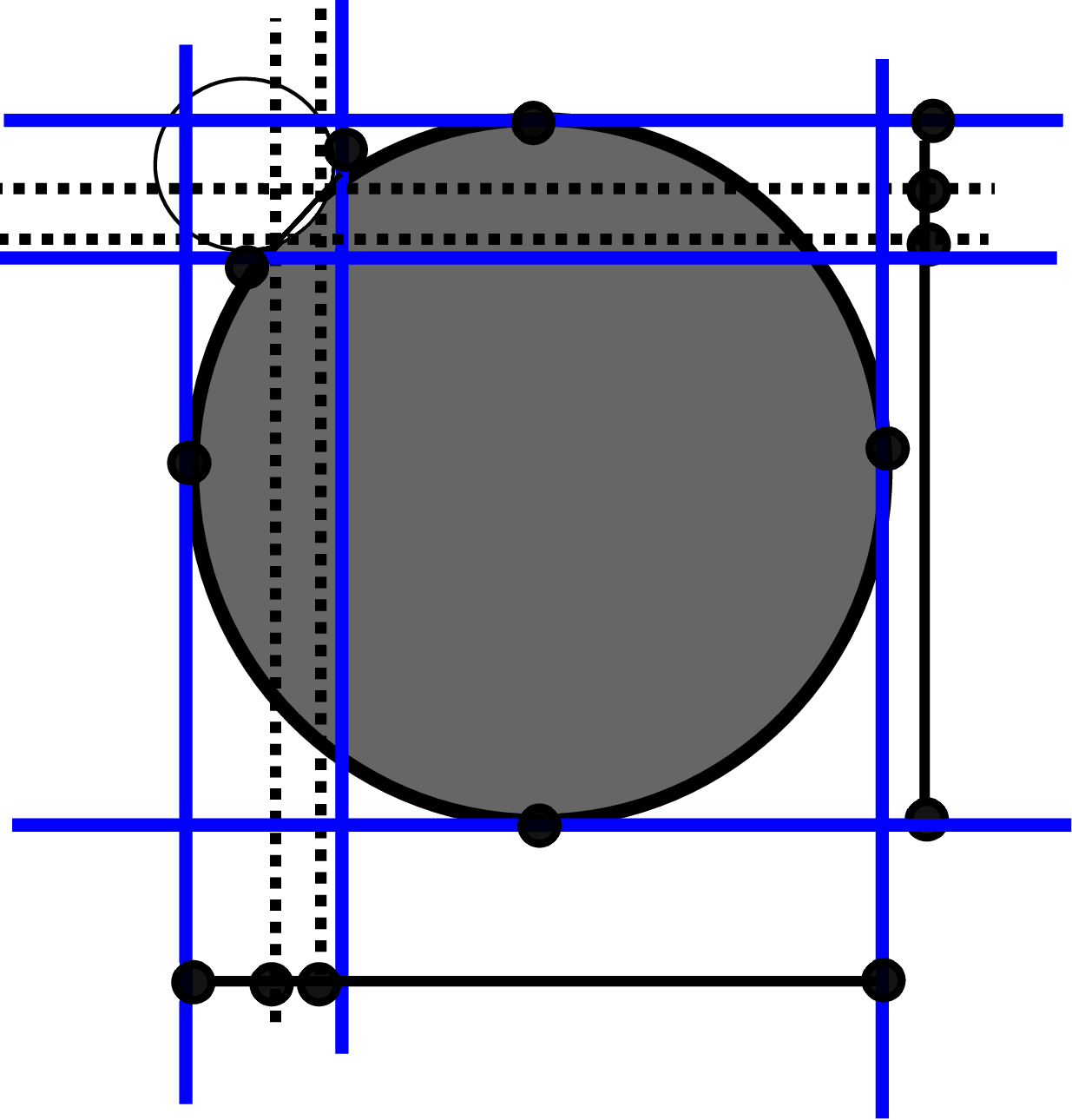}
	\caption{Example \ref{ex:2}. Dots are for horizontal poles and vertical poles of circles and vertices of Poincar\'e-Reeb graphs for example. Blue colored lines show horizontal lines and vertical lines related to this case.}
	\label{fig:8}
\end{figure}

However, this is not an MBCC arrangement. More precisely, we can also check that the Poincar\'e-Reeb graphs of the new bounded connected component $D_{{\mathcal{S}}^{\prime}}$ are graphs with exactly four vertices and homeomorphic to a closed interval. We cannot have such a case of Poincar\'e-Reeb graphs by considering MBCC arrangements only by applying Theorem \ref{thm:2} for example. More explicitly, to change a given Poincar\'e-Reeb graph, homeomorphic to a closed interval, to another Poincar\'e-Reeb graph homeomorphic to a closed interval, we must consider the case of Theorem \ref{thm:2} (\ref{thm:2.3}) (Theorem \ref{thm:2} (\ref{thm:2.3.1}) or (\ref{thm:2.3.3})). However, this forces us to apply the case of Theorem \ref{thm:2} (\ref{thm:2.2}) to the other Poincar\'e-Reeb graph.
\end{Ex}

As an explicit case for generalized cases of Theorem \ref{thm:2} (\ref{thm:2.2}) 
we cannot discuss by Theorem \ref{thm:2} (\ref{thm:2.2}), we present FIGURE \ref{fig:9}.
\begin{figure}
\includegraphics[width=80mm,height=50mm]{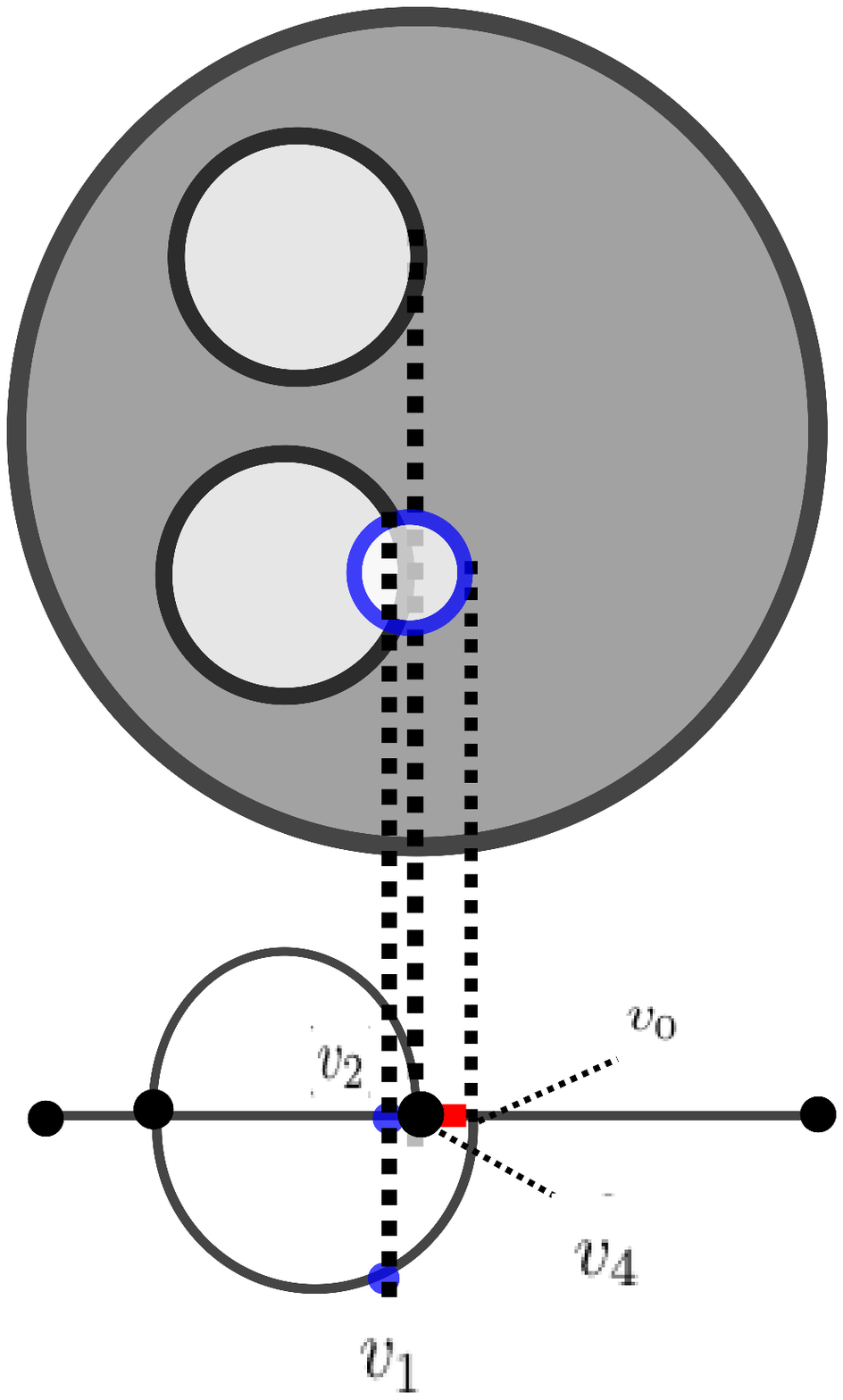}
\caption{A simplest case for general cases of Theorem \ref{thm:2} (\ref{thm:2.2}). By deleting vertices colored in blue and collapsing the edge depicted as a small segment in red, we have a graph isomorphic to our original Poincar\'e-Reeb V-digraph. In Theorem \ref{thm:5}, it is discussed again and we add the notation there to this figure for vertices.}
\label{fig:9}
\end{figure}
\subsection{Additional Comments.}
We present several additional problems and comments.
\begin{Prob}
	In Theorem \ref{thm:2}, can we solve the case (\ref{thm:2.2}) explicitly dropping the additional assumption on the vertical lines and the vertical segments?
\end{Prob}
This is also related to FIGURE \ref{fig:9}.

We discuss this.

In Theorem \ref{thm:2} (\ref{thm:2.2}) with the condition on the vertical lines and the vertical segments being dropped, let $p=(x_{j^{\prime},1}, x_{j^{\prime},2})$ be a vertical pole such that $p=(t,x_{j^{\prime},2})$ is outside the closure $\overline{D_{\mathcal{S}}}$ for $t=x_{j^{\prime},1}-\epsilon$ for any sufficiently small $\epsilon>0$. We call such a vertical pole a vertical pole of {\it I-type}. We call a vertical pole which is not a vertical pole of I-type a vertical pole of {\it II-type}.
 
Let $p=(x_{j^{\prime},1}, x_{j^{\prime},2})$ be a I-type vertical pole.
For a vertical line passing through $p$, consider the intersection with $\overline{D_{\mathcal{S}}}$ and the connected component containing $p$. This is represented by the form $\{(x_{j^{\prime},1},t) \mid x_{j^{\prime},2,{\rm l}} \leq t \leq x_{j^{\prime},2,{\rm u}}\}$ for real numbers satisfying the relation $x_{j^{\prime},2,{\rm l}}<x_{j^{\prime},2}<x_{j^{\prime},2,{\rm u}}$.

We can find an increasing sequence $\{x_{j^{\prime},2,{\rm l},j^{\prime \prime}}\}_{j^{\prime \prime}=1}^a$ ($a \geq 1$) and another increasing sequence $\{x_{j^{\prime},2,{\rm r},j^{\prime \prime}}\}_{j^{\prime \prime}=1}^b$ ($b \geq 1$) satisfying the following conditions.
\begin{enumerate}
\item The relation $x_{j^{\prime},2,{\rm l}}<x_{j^{\prime},2,{\rm l},j^{\prime \prime}} \leq x_{j^{\prime},2,{\rm u}}$ and $x_{j^{\prime},2,{\rm l},a}=x_{j^{\prime},2,{\rm u}}$.
\item The relation $x_{j^{\prime},2,{\rm l}}<x_{j^{\prime},2,{\rm r},j^{\prime \prime}} \leq x_{j^{\prime},2,{\rm u}}$ and $x_{j^{\prime},2,{\rm r},b}=x_{j^{\prime},2,{\rm u}}$.
\item The intersection of these two sequences is the single set consisting of $x_{j^{\prime},2,{\rm u}}$.
\item For some integer $j^{\prime \prime}={j_0}^{\prime \prime} \neq a$, $x_{j^{\prime},2,{\rm l},{j_0}^{\prime \prime}}=x_{j^{\prime},2}$.
\item We can represent the set of all edges entering $v_0:=q_{D_{\mathcal{S}},1}(p)$ by $\{e_{v_0,{\rm l},j^{\prime \prime}}\}_{j^{\prime \prime}=1}^a$. Each set ${q_{D_{\mathcal{S}},1}}^{-1}(e_{v_0,{\rm l},j^{\prime \prime}}) \bigcap {{\pi}_{2,1,1}}^{-1}(x_{j^{\prime},1})$ is the vertical segment $\{(x_{j^{\prime},1},t) \mid x_{j^{\prime},2,{\rm l},j^{\prime \prime}-1} \leq t \leq  x_{j^{\prime},2,{\rm l},j^{\prime \prime}}\}$ where $x_{j^{\prime},2,{\rm l},0}:=x_{j^{\prime},2,{\rm l}}$. The set
$\{(x_{j^{\prime},1},x_{j^{\prime},2,{\rm l},j^{\prime \prime}})\}_{j^{\prime \prime}=1}^{a-1}$
 is the set of all vertical poles of type I in the segment $\{(x_{j^{\prime},1},t) \mid x_{j^{\prime},2,{\rm l}} \leq t \leq x_{j^{\prime},2,{\rm u}}\}$.
\item We can represent the set of all edges departing from $v_0$ by $\{e_{v_0,{\rm r},j^{\prime \prime}}\}_{j^{\prime \prime}=1}^b$. Each set ${q_{D_{\mathcal{S}},1}}^{-1}(e_{v_0,{\rm r},j^{\prime \prime}}) \bigcap {{\pi}_{2,1,1}}^{-1}(x_{j^{\prime},1})$ is the vertical segment $\{(x_{j^{\prime},1},t) \mid x_{j^{\prime},2,{\rm r},j^{\prime \prime}-1} \leq t \leq  x_{j^{\prime},2,{\rm r},j^{\prime \prime}}\}$ where $x_{j^{\prime},2,{\rm r},0}:=x_{j^{\prime},2,{\rm l}}$. The set $\{(x_{j^{\prime},1},x_{j^{\prime},2,{\rm r},j^{\prime \prime}})\}_{j^{\prime \prime}=1}^{b-1}$
 is the set of all vertical poles of type II in the segment $\{(x_{j^{\prime},1},t) \mid x_{j^{\prime},2,{\rm l}} \leq t \leq x_{j^{\prime},2,{\rm u}}\}$. 
 \item Points in the segment $\{(x_{j^{\prime},1},t) \mid x_{j^{\prime},2,{\rm l}} \leq t \leq x_{j^{\prime},2,{\rm u}}\}$ are, except points of the form $(x_{j^{\prime},1},x_{j^{\prime},2,{\rm l},j^{\prime \prime}})$ or $(x_{j^{\prime},1},x_{j^{\prime},2,{\rm r},j^{\prime \prime}})$, not vertical poles, horizontal poles, or points contained in two distinct circles from $\mathbb{S}$.
\end{enumerate}

We label the vertex $v_0:=q_{D_{\mathcal{S}},1}(p)$ by a suitable real number $i(v_0)$ greater than the original value $m_{D_{\mathcal{S}},1}(v_0)$ and sufficiently close to it. For each of the two edges $e_{v_0,{\rm l},{j_0}^{\prime \prime}}$ and $e_{v_0,{\rm l},{j_0}^{\prime \prime}+1}$, entering $v_0$, we add a new vertex $v_1$ and $v_2$, respectively and label them by a suitable same real number $i(v_1)=i(v_2)$ smaller than $m_{D_{\mathcal{S}},1}(v_0)$ and sufficiently close to this. We add another new vertex $v_3$ between $v_1$ and $v_0$ and label $v_3$ by the original value $m_{D_{\mathcal{S}},1}(v_0)$ if ${j_0}^{\prime \prime} \neq 1$. We add another new vertex $v_4$ between $v_2$ and $v_0$ and label $v_4$ by the original value $m_{D_{\mathcal{S}},1}(v_0)$ if ${j_0}^{\prime \prime} \neq a$. 

For the remaining edges $e_{v_0,{\rm l},j^{\prime \prime}}$ with $j^{\prime \prime} <{j_0}^{\prime \prime}$ and $e_{v_0,{\rm l},j^{\prime \prime}}$ with ${j}^{\prime \prime}>{j_0}^{\prime \prime}+1$ we change them into edges entering $v_3$ and $v_4$, respectively. 
We have the unique integer $j^{\prime \prime}={j_0}^{\prime \prime \prime}$ satisfying the relation $x_{j^{\prime},2,{\rm r},{j_0}^{\prime \prime \prime}-1}<x_{j^{\prime},2}<x_{j^{\prime},2,{\rm r},{j_0}^{\prime \prime \prime}}$.
For the remaining edges $e_{v_0,{\rm r},j^{\prime \prime}}$ with $j^{\prime \prime} <{j_0}^{\prime \prime \prime}$ and $e_{v_0,{\rm r},j^{\prime \prime}}$ with ${j}^{\prime \prime}>{j_0}^{\prime \prime \prime}$, we change them into edges entering $v_3$ and $v_4$, respectively. 


Through this, we can also check the following immediately.

\begin{Thm}
\label{thm:5}
We discuss Theorem \ref{thm:2} {\rm (}\ref{thm:2.2}{\rm )} with the condition on the vertical lines and the vertical segments being dropped. Let $p=(x_{j^{\prime},1}, x_{j^{\prime},2})$ a vertical pole of I-type. 
The V-digraph obtained above is isomorphic to the resulting Poincar\'e-Reeb V-digraph $W_{D_{\mathcal{S^{\prime}}},1}$. 
\end{Thm}
FIGURE \ref{fig:9} is for the case $(a,b,{j_0}^{\prime \prime},{j_0}^{\prime \prime \prime})=(3,1,1,1)$. We have the case of a vertical pole $p$ of II-type by considering Proposition \ref{prop:2} (\ref{prop:2}). By considering Proposition \ref{prop:2} (\ref{prop:2.3}), we also have the case for the Poincar\'e-Reeb V-digraph $W_{D_{\mathcal{S^{\prime}}},2}$. 

\begin{Prob}
Can we define and investigate fundamental properties of arrangements of circles (MBC arrangements) respecting Example \ref{ex:2}, and so on?
\end{Prob}
We have some positive explicit idea. We also need to check rigorously. If you readers have another positive idea, it is nice and it should be presented in some way.  

We can generalize this problem as follows. 
\begin{Prob}
Can we define and investigate fundamental properties of meaningful classes of arrangements of circles (MBC arrangements) further?
\end{Prob}

\begin{Prob}
Can we apply our present study and Problems to obtain explicit classes of explicit real algebraic functions, maps and manifolds. 
\end{Prob}
As presented in our introduction for example, we are interested in constructing explicit real algebraic functions and manifolds. For example, we are interested in certain classes generalizing the canonical projections of the unit spheres $S^{m}$, obtained through Example \ref{ex:1} with \cite{kitazawa7}. Our interest lies in the following problem, started by the author.
\begin{Prob}
From a graph, can we reconstruct a nice real algebraic function whose Reeb graph is a given graph?
\end{Prob}
 \cite{kitazawa5} is a pioneering study. This considers construction on real algebraic maps onto regions surrounded by real algebraic curves or hypersurfaces which are mutually disjoint. A closely related and independent study 	\cite{bodinpopescupampusorea} has helped us to present \cite{kitazawa5}. This study is on finding a region surrounded by a non-singular real algebraic curves and collapses to a nicely embedded graph in the plane: approximation is a key-ingredient.
 A Poincar\'e-Reeb graph of an region is first introduced in \cite{bodinpopescupampusorea, sorea1, sorea2} as a graph the region naturally collapses to. 
 Our study obtains a real algebraic function obtained as the composition of the presented real algebraic map with a projection and information of its Reeb graph. \cite{kitazawa6} is a related article in a proceeding (with a short review). \cite{kitazawa7} extends \cite{kitazawa5} where \cite{kitazawa5} considers the case the hypersurfaces are disjoint.

For related study, we should review reconstruction of nice differentiable (smooth) functions whose Reeb graphs are prescribed graphs and our real algebraic studies follow these studies. \cite{sharko} is a pioneering study and \cite{masumotosaeki, michalak} are important related studies. The author published related studies such as \cite{kitazawa4}. 
In these studies, local reconstruction or reconstruction of local functions onto neighborhoods of vertices of graphs are main ingredients. After that, we glue these local functions together. The constraint that functions are real algebraic makes the problems difficult and for example, it is very difficult to find nice classes of graphs in real algebraic cases.  
\section{Conflict of interest and Data availability.}
\noindent {\bf Conflict of interest.} \\
The author was a member of the project JSPS Grant Number JP17H06128.
The author was also a member of the project JSPS KAKENHI Grant Number JP22K18267. Principal Investigator for them is all Osamu Saeki.  The author works at Institute of Mathematics for Industry (https://www.jgmi.kyushu-u.ac.jp/en/about/young-mentors/) and this is closely related to our study. Our study thanks them for the supports. The author is also a researcher at Osaka Central
Advanced Mathematical Institute (OCAMI researcher), supported by MEXT Promotion of Distinctive Joint Research Center Program JPMXP0723833165. Although he is not employed there, our study also thanks this. \\
\ \\
{\bf Data availability.} \\
Data essentially supporting our present study are all in the
 paper.


\begin{thebibliography}{25}
	%

	\bibitem{bochnakcosteroy} J. Bochnak, M. Coste and M.-F. Roy, \textsl{Real algebraic geometry}, Ergebnisse der Mathematik und ihrer Grenzgebiete (3) [Results in Mathematics and Related Areas (3)], vol. 36, Springer-Verlag, Berlin, 1998. Translated from the 1987 French original; Revised by the authors.
	\bibitem{bodinpopescupampusorea} A. Bodin, P. Popescu-Pampu and M. S. Sorea, \textsl{Poincar\'e-Reeb graphs of real algebraic domains}, Revista Matem\'atica Complutense, https://link.springer.com/article/10.1007/s13163-023-00469-y, 2023, arXiv:2207.06871v2.
\bibitem{carmesinschulz} S. Carmesin and A. Schulz, \textsl{Arrangements of orthogonal circles with many intersections}, Graph Drawing and Network Visualization (a conference paper), SPRINGER NATURE Link, 2021, arXiv:2106.03557v2. 
\bibitem{golubitskyguillemin} M. Golubitsky and V. Guillemin, \textsl{Stable Mappings and Their Singularities}, Graduate Texts in Mathematics (14), Springer-Verlag (1974).
\bibitem{kitazawa1} N. Kitazawa, \textsl{On round fold maps} (in Japanese), RIMS Kokyuroku Bessatsu B38 (2013), 45--59.
\bibitem{kitazawa2} N. Kitazawa, \textsl{On manifolds admitting fold maps with singular value sets of concentric spheres}, Doctoral Dissertation, Tokyo Institute of Technology (2014).
\bibitem{kitazawa3} N. Kitazawa, \textsl{Fold maps with singular value sets of concentric spheres}, Hokkaido Mathematical Journal Vol.43, No.3 (2014), 327--359.
\bibitem{kitazawa4} N. Kitazawa, \textsl{On Reeb graphs induced from smooth functions on $3$-dimensional closed orientable manifolds with finitely many singular values}, Topol. Methods in Nonlinear Anal. Vol. 59 No. 2B, 897--912, arXiv:1902.08841.
\bibitem{kitazawa5} N. Kitazawa, \textsl{Real algebraic functions on closed manifolds whose Reeb graphs are given graphs}, Methods of Functional Analysis and Topology Vol. 28 No. 4 (2022), 302--308, arXiv:2302.02339, 2023.
\bibitem{kitazawa6} N. Kitazawa, \textsl{Explicit construction of explicit real algebraic functions and real algebraic manifolds via Reeb graphs}, Algebraic and geometric methods of analysis 2023 “The book of abstracts”, 49—51, this is the abstract book of the conference "Algebraic and geometric methods of analysis 2023" and published after a short review (https://www.imath.kiev.ua/$\sim$topology/conf/agma2023/), https://imath.kiev.ua/$\sim$topology/conf/agma2023/contents/abstracts/texts/kitazawa/kitazawa.pdf, 2023.



\bibitem{kitazawa7} N. Kitazawa, \textsl{Reconstructing real algebraic maps locally like moment-maps with prescribed images and compositions with the canonical projections to the $1$-dimensional real affine space}, the title has changed from previous versions, arXiv:2303.10723, 2024.
\bibitem{kohnpieneranestadrydellshapirosinnsoreatelen} K. Kohn, R. Piene, K. Ranestad, F. Rydell, B. Shapiro, R. Sinn, M-S. Sorea and S. Telen, \textsl{Adjoints and Canonical Forms of Polypols}, to appear in Documenta Mathematica, arXiv:2108.11747.
\bibitem{kollar} J. Koll\'ar, \textsl{Nash's work in algebraic geometry}, Bulletin (New Series) of the American Mathematical Society (2) 54, 2017, 307--324.
\bibitem{masumotosaeki} Y. Masumoto and O. Saeki, \textsl{A smooth function on a manifold with given Reeb graph}, Kyushu J. Math. 65 (2011), 75--84.
\bibitem{michalak} L. P. Michalak, \textsl{Realization of a graph as the Reeb graph of a Morse function on a manifold}. Topol. Methods in Nonlinear Anal. 52 (2) (2018), 749--762, arXiv:1805.06727.
\bibitem{milnor} J. Milnor, \textsl{Lectures on the h-cobordism theorem}, Math. Notes, Princeton Univ. Press, Princeton, N.J. 1965.
\bibitem{reeb} G. Reeb, \textsl{Sur les points singuliers d\'{}une forme de Pfaff compl\'{e}tement int\`{e}grable ou d\'{}une fonction num\'{e}rique}, Comptes Rendus
 Hebdomadaires des S\'{e}ances de I\'{}Acad\'{e}mie des Sciences 222 (1946), 847--849.
\bibitem{saeki1} O. Saeki, \textsl{Reeb spaces of smooth functions on manifolds}, International Mathematics Research Notices, maa301, Volume 2022, Issue 11, June 2022, 3740--3768, https://doi.org/10.1093/imrn/maa301.

\bibitem{saeki2} O. Saeki, \textsl{Reeb spaces of smooth functions on manifolds II}, Res. Math. Sci. 11, article number 24 (2024), https://link.springer.com/article/10.1007/s40687-024-00436-z.
\bibitem{sharko} V. Sharko, \textsl{About Kronrod-Reeb graph of a function on a manifold}, Methods of Functional Analysis and
 Topology 12 (2006), 389--396.
\bibitem{sorea1} M. S. Sorea, \textsl{The shapes of level curves of real polynomials near strict local maxima},  Ph. D. Thesis, Universit\'e de Lille, Laboratoire Paul Painlev\'e, 2018.
\bibitem{sorea2} M. S. Sorea, \textsl{Measuring the local non-convexity of real algebraic curves}, Journal of Symbolic Computation 109 (2022), 482--509.
\bibitem{tamaki1} D. Tamaki, Algebraic Topology A Guide to literature,  http://pantodon.jp/index.rb?body=about, 2023. 
\bibitem{tamaki2} D. Tamaki, Algebraic Topology A Guide to literature (Submanifold arrangement), http://pantodon.jp/index.rb?body=submanifold\_arrangement, 2023.
\bibitem{tamaki3} D. Tamaki, Algebraic Topology A Guide to literature (Arrangement variations), http://pantodon.jp/index.rb?body=arrangement\_variations, 2023.


\end{thebibliography}
\end{document}